\documentclass{article}
\usepackage[letterpaper,top=2cm,bottom=2cm,left=3cm,right=3cm,marginparwidth=1.75cm]{geometry}
\usepackage[table]{xcolor}
\usepackage{amsthm}
\usepackage{amsmath}
\usepackage{graphicx}
\usepackage{amssymb}
\usepackage{bm}
\usepackage[colorlinks=true, allcolors=blue]{hyperref}
\usepackage{diagbox}
\usepackage{tikz}
\usetikzlibrary{positioning}
\usepackage[labelfont=bf]{caption}
 \usepackage{float}

\DeclareMathOperator*{\argmin}{arg\,min}

\usepackage{subcaption}
\usepackage{algorithm}
\usepackage{algpseudocode}
\newtheorem{lemma}{Lemma}
\newtheorem{corollary}{Corollary}
\newtheorem{theorem}{Theorem}
\usepackage{breakurl}

\title{A Generalized Alternating Anderson Acceleration Method}
\author{Yunhui He\thanks{Department of Mathematics, University of Houston, 3551 Cullen Blvd, Houston, Texas 77204-3008, USA. {\tt{yhe43@central.uh.edu},  \tt{santolo.leveque@matfyz.cuni.cz}}. Present address of Santolo Leveque:  Sokolovsk\'{a} 49/83, Department of Numerical Mathematics, Faculty of Mathematics and Physics, Charles University, 186 75, Praha 8, Czech Republic.} 
	\and Santolo Leveque\footnotemark[1]}
 
\begin{document}
\maketitle
%
\begin{abstract}
In this work, we propose a generalized alternating Anderson acceleration method, a periodic scheme composed of $t$ fixed-point iteration steps, interleaved with $s$ steps of Anderson acceleration with window size $m$, to solve linear and nonlinear problems. This allows flexibility to use different combinations of fixed-point iteration and Anderson iteration.  We present a convergence analysis of the proposed scheme for accelerating the Richardson iteration in the linear case,  with a focus on specific parameter choices of interest. Specifically, we prove convergence of the proposed method under contractive fixed-point iteration and provide a sufficient condition for convergence when the Richardson iteration matrix is diagonalizable and noncontractive. To demonstrate the broader applicability of our proposed method, we use it to accelerate Jacobi iteration, Picard iteration, gradient descent, and the alternating direction method of multipliers in solving partial differential equations and nonlinear, nonsmooth optimization problems.  The numerical results illustrate that the proposed scheme is more efficient than the existing windowed Anderson acceleration and alternating Anderson ($s=1$) in terms of iteration number and CPU time for careful choice of parameters $m, s, t$. 
\end{abstract}

{\bf Keywords.} Alternating Anderson acceleration, fixed point iteration, convergence analysis, ADMM

\vspace{3mm}

{\bf MSC codes.} 65F10, 65H10, 65K10   

 \section{Introduction}
 Anderson acceleration (AA) or Anderson mixing \cite{anderson1965iterative} was first proposed by Donald G. Anderson in 1965 to improve the performance of fixed-point iterations in solving systems of nonlinear equations. Later, AA is applied to solve both linear and nonlinear problems. The main idea of AA is to use previous iterates to generate a new approximation.  Overtime, AA has been widely used in many areas, such as fluid flow \cite{both2019anderson,pollock2019anderson}, optimization problems \cite{wang2021asymptotic,bian2022anderson,peng2018anderson,feng2024anderson}, reinforcement learning \cite{shi2019regularized,sun2021damped,wang2023regularization,geist2018anderson}, and so on \cite{an2017anderson,higham2016anderson,ji2023improved}. Under certain conditions, when AA using all previous iteration information in combination with Richardson iteration is applied to solving linear systems, its iterates can be recovered from the GMRES method; see \cite{walker2011anderson,potra2013characterization}. Recently,  there have been many studies on the convergence analysis of AA \cite{potra2013characterization,walker2011anderson,toth2015convergence,evans2020proof,sterck2021asymptotic,de2022linear,bian2021anderson,brezinski2018shanks} and its development \cite{suryanarayana2019alternating,pratapa2016anderson,chen2024non,chen2022composite}.   A recent review on AA can be found in \cite{anderson2019comments}.
 
 In practice, we consider AA with finite window size, $m$, referred to as AA($m$). In AA($m$), each update is a linear combination of the underlying fixed-point iteration acting on previous $m+1$ iterates. The combination coefficients change at each iteration and are obtained by solving a least-squares problem. To reduce computational cost, one may consider a different combination of AA and fixed-point iteration.  In \cite{pratapa2016anderson}, AA is employed at periodic intervals
 within Jacobi iteration.  Later, \cite{suryanarayana2019alternating} applies this strategy to Richardson iteration, where a preconditioner is considered. \cite{lupo2019convergence} further provides a convergence analysis of the alternating Anderson with Jacobi iteration when the underlying fixed-point iteration is contractive and makes a connection with GMRES for full alternating Anderson, i.e., using all previous iteration information. A convergence analysis of alternating Anderson-Picard, in which an AA($m$) step is applied after $m$ steps of Picard iterations, for nonlinear problems, is provided in \cite{feng2024convergence}. In addition, the authors in \cite{feng2024convergence} establish the equivalence between the alternating Anderson-Picard and a multisecant-GMRES method that employs GMRES to solve a multisecant linear system at each iteration, and demonstrate  the efficiency and robustness of the alternating Anderson-Picard. 
 In \cite[Algorithm 7]{feng2024anderson}, the author considers a Generalized Alternate
 Anderson-Picard (GAAP) method, which allows flexible configurations of the Picard step
 $t$ and the Anderson acceleration history $m$.  In this work, we extend the alternating strategy to a general mixing scheme, where we consider a periodic pattern comprising $t$ steps of fixed-point iteration interleaved with $s$ steps of AA($m$). We refer to this approach as aAA($m$)[$s$]--FP[$t$] or simply aAA--FP. Here, $m, s, t$ are integers.  Note that the GAAP method is aAA($m$)[1]--FP[$t$]. In \cite{banerjee2016periodic}, it was found that taking $m\geq 2t$ generally improves performance of aAA($m$)[1]--FP[$t$] for the problems considered there. Our aAA--FP method provides us with more flexibility for general models compared to the GAAP method. In this work, we investigate the selection of parameters $m$, $s$, and $t$ to improve the performance of aAA--FP.
 
 The contributions of this work are as follows. 
 \begin{itemize}
 	\item First, we propose a generalized alternating Anderson mixing scheme aAA($m$)[$s$]--FP[$t$], which allows for flexible choices of $m, s$, and $t$. It is an extension of a simple alternating Anderson.
 	\item Second, in the linear case, we establish a connection between aAA($\infty$)[1]--FP[$t$] and GMRES for the Richardson fixed-point iteration. Furthermore, we present the convergence analysis of aAA($m$)[$s$]--FP[$t$] when the Richardson fixed-point iteration is contractive. In addition, we provide a sufficient condition for the convergence of aAA($m$)[$s$]--FP[$t$] when the Richardson iteration matrix is diagonalizable and noncontractive.
 	\item Third, to highlight its wide-ranging utility, our method is employed to enhance the efficiency of Jacobi iteration, Gauss--Seidel iteration, Picard iteration, gradient descent, and the alternating direction method of multipliers for solving partial differential equations and nonlinear optimization problems. We provide numerical evidence of the efficiency of aAA--FP in terms of both iteration count and CPU time, compared to the standalone fixed-point iteration, GMRES,  and aAA($m$)[$1$]--FP[$t$]. The results show that, for suitable choices of the parameters, $m, s, t$, aAA($m$)[$s$]--FP[$t$] is competitive with, or outperforms, the other methods.
 \end{itemize}
 
 Regarding the numerical results, we conduct a wide range of test problems, including nonsymmetric linear systems, Navier--Stokes equations, and nonlinear, nonsmooth optimization problems. For the highly ill-conditioned, nonsymmetric linear systems, we consider aAA--FP, GMRES, windowed AA, and AA($\infty$) that are applied to accelerate the noncontractive Jacobi and contractive Gauss--Seidel iterations. Our numerical results show that aAA--FP with Gauss-Seidel outperforms the Jacobi-based aAA--FP version.   For the Navier--Stokes problem, we accelerate the Picard iteration by aAA--FP and AA. We see that aAA--FP significantly speeds up the performance of the Picard iteration. Moreover, aAA--FP can be better than AA($\infty$). For the nonlinear optimization problems studied here, we consider the Alternating Direction Method of Multipliers (ADMM) as a fixed-point iteration. Our study demonstrates that aAA--FP and AA accelerate ADMM by approximately an order of magnitude, resulting in a reduction by a factor of 10 in both CPU time and iteration count. We also note that for some careful choices of $s$ and $t$, aAA($m$)[$s$]--FP[$t$] outperforms AA($m$) and is competitive as AA($\infty$). Finally, we accelerate the gradient descend (GD) in the solution of the regularized logistic regression. Also for this problem, we observe that AA and aAA--FP can drastically speed-up the solution process by an order of magnitude, with aAA--FP being able to outperform AA for suitable choice of the parameters. In conclusion, our aAA($m$)[$s$]--FP[$t$] method has the potential to solve challenging linear and nonlinear problems for which GMRES is not applicable in the nonlinear case, and for which the underlying fixed-point iteration may converge very slowly or even diverge.
 
 The remainder of this work is organized as follows. In Section \ref{sec:aAAFP}, we propose our generalized alternating Anderson acceleration method. Then, the convergence analysis is provided in Section \ref{sec:convergence}. To illustrate the efficiency of our proposed method, we conduct various examples in Section \ref{sec:num}. Finally, we draw conclusions in Section \ref{sec:con}.
 
 \section{Generalized alternating Anderson acceleration}\label{sec:aAAFP}
 In this section, we propose a generalized alternating Anderson acceleration method that combines Anderson iteration and fixed-point iteration in a novel mixing scheme.
 Assume that we solve the following problem
 \begin{displaymath}
 	g(\mathbf{x})=0,
 \end{displaymath}
 where $\mathbf{x} \in \mathbb{R}^n$, $g: \mathbb{R}^n\rightarrow \mathbb{R}^n$, and $\mathbf{x}^*$ is the exact solution.  Suppose that we use the following fixed-point iteration to solve the above problem
 \begin{equation}\label{eq:FP}
 	\mathbf{x}_{k+1}=q(\mathbf{x}_k),
 \end{equation}
 where $\mathbf{x}^*$ is a fixed point, i.e., $\mathbf{x}^*=q(\mathbf{x}^*)$. Define the $k$th residual as
 \begin{equation}\label{eq:kthresidual}
 	r(\mathbf{x}_k)=q(\mathbf{x}_k)-\mathbf{x}_k.
 \end{equation}
 
 AA can be used to improve the performance of \eqref{eq:FP}. In practice, we often consider AA with finite window size $m$, that is,  AA($m$). We first present the framework of AA($m$) in Algorithm \ref{alg:AAm}.  We allow $m=\infty$, and denote the corresponding AA as AA($\infty$). In AA, each iterate is updated by \eqref{eq:xkp1-AA}, which is a linear combination of previous $m_k+1$ fixed-point iterates. In each step, the combination coefficients in \eqref{eq:xkp1-AA} are obtained by solving the least-squares problem \eqref{eq:min-AA}. Note that when $m=0$, AA(0) is reduced to the original fixed-point iteration \eqref{eq:FP}.  In \eqref{eq:min-AA}, the norm $\|\cdot\|$ is denoted as the 2-norm. One may consider other norms; see, for example, \cite{toth2015convergence,yang2022anderson}. 
 \begin{algorithm}[H] 
 	\caption{Anderson Acceleration with finite window size $m$ (AA$(m)$)}\label{alg:AAm}
 	\begin{algorithmic}[1] 
 		\State{Given $\mathbf{x}_0$  and $m\geq0$}
 		\For {$k=0,1,\cdots$ until convergence }
 		\begin{equation}\label{eq:xkp1-AA} 
 			\mathbf{x}_{k+1} = q(\mathbf{x}_k) + \sum_{i=1}^{m_k}\gamma_i^{(k)} \left(q(\mathbf{x}_k)-q(\mathbf{x}_{k-i}) \right),
 		\end{equation}
 		where $m_k=\min\{m,k\}$ and $\gamma_i^{(k)}$ is obtained by solving the following least-squares problem
 		\begin{equation}\label{eq:min-AA} 
 			\min_{\left(\gamma_1^{(k)},\cdots, \gamma_{m_k}^{(k)}\right)} \left\|r(\mathbf{x}_k)+\sum_{i=1}^{m_k} \gamma_i^{(k)} \left(r(\mathbf{x}_k)-r(\mathbf{x}_{k-i}) \right)\right\|^2,
 		\end{equation}
 		where $r(\mathbf{x})=q(\mathbf{x})-\mathbf{x}$.
 		\EndFor
 	\end{algorithmic}
 \end{algorithm}

 For the convenience of convergence analysis, we rewrite Algorithm \ref{alg:AAm} as Algorithm  \ref{alg:AA-alt}. 
 \begin{algorithm}
 	\caption{Alternative format of Anderson acceleration with window size $m$ (AA($m$))} \label{alg:AA-alt}
 	\begin{algorithmic}[1] 
 		\State Given $x_0$ and $m\geq 0$ 
 		\For {$k=0,1,2,\cdots$ until convergence}
 		\begin{equation}\label{AA-alt-update}
 			\mathbf{x}_{k+1}=q(\mathbf{x}_k)+\sum_{i=1}^{m_k} \tau^{(k)}_i\left(q(\mathbf{x}_{k-i+1})- q(\mathbf{x}_{k-i})\right),
 		\end{equation}
 		where $m_k=\min\{m,k\}$ and $\tau_i^{(k)}$ is obtained by solving the following least-squares problem
 		\begin{equation}\label{eq:AA-LSQ-alt}
 			\min_{\left(\tau_1^{(k)},\cdots, \tau_{k}^{(k)}\right)}\left\|r(\mathbf{x}_k) + \sum_{i=1}^{m_k} \tau^{(k)}_i\left(r(\mathbf{x}_{k-i+1})- r(\mathbf{x}_{k-i})\right)\right\|^2,
 		\end{equation}
 		where $r(\mathbf{x})=q(\mathbf{x})-\mathbf{x}$.
 		\EndFor
 	\end{algorithmic}
 \end{algorithm}
 It can be shown that $\tau_j^{(k)}=\sum_{i=j}^{m_k}\gamma_i^{(k)}$.

 In the literature, many efforts have been made to improve the performance of AA. For example, in \cite{chen2022composite}, the authors 
 propose a systematic way to dynamically alternate the window size $m_1$ by the multiplicative composite combination, i.e.,  applying AA($m_1$) in the outer loop and applying AA($m_2$) in the inner loop. In \cite{chen2024non}, the authors consider a nonstationary Anderson acceleration algorithm with optimized damping in each iteration by applying one extra inexpensive optimization, i.e, adding another linear combination of previous $\{\mathbf{x}_{k-i}\}_{i=0}^{m_k}$ in \eqref{eq:xkp1-AA}. Rather than applying AA at every step, the authors in \cite{pratapa2016anderson} employ AA at periodic intervals within the Jacobi iteration, named alternating Anderson-Jacobi method (AAJ), which outperforms the GMRES method of problems considered.  Subsequently, the alternating Anderson-Richardson (AAR) method is extended to incorporate preconditioning, address efficient parallel implementation, and provide serial MATLAB and parallel C/C++ implementations \cite{suryanarayana2019alternating}.  The convergence analysis of AAR method is presented in \cite{lupo2019convergence}, along with a discussion of its connection to GMRES. We note that \cite{lupo2019convergence} uses the framework of Algorithm \ref{alg:AA-alt} to define AAR.
 
 This work extends the alternating strategy discussed in \cite{pratapa2016anderson,suryanarayana2019alternating,lupo2019convergence,feng2024convergence,feng2024anderson} to a novel mixing scheme. A detailed description of the proposed algorithm is presented below. Let $m,s,t$ be integers.   We consider a combination of $t$ steps of fixed-point iteration followed by $s$ steps of AA($m$), and then repeat this procedure. We denote the corresponding method as aAA($m$)[$s$]--FP[$t$], or simply aAA--FP when parameter specification is unnecessary.   The strategy loops until a certain convergence criterion in the residual is satisfied. A pseudo-code of the strategy is given in Algorithm \ref{saAAm_FPn}. Here, we use MATLAB's notations. Specifically, given an $n \times k$ matrix $\mathbf{Y}$, we denote with $\mathbf{Y}(:, 1:d)$ the first $d$ columns of $\mathbf{Y}$. In addition, $\mathrm{ones}(1,m_k)$ represents the vector of length $m_k$ of all ones, and $\mathrm{kron}(\mathbf{e}_k, \mathbf{\bar{x}}_k)$ represents the Kronecker product between $\mathbf{e}_k$ and $\mathbf{\bar{x}}_k$.

 In aAA($m$)[$s$]--FP[$t$] when $t=0$, we recover AA($m$), for any given $s$. When $m=s=t=1$, aAA($m$)[$s$]--FP[$t$] is restarted AA(1), which has been studied in \cite{both2019anderson,krzysik2025asymptotic}. We note that when $s=1$, our aAA($m$)[$s$]--FP[$t$] is not equivalent to the AAR presented in \cite[Algorithm 1]{lupo2019convergence}. In fact,  the AAR method uses $s=1$, $p=s+t$, defining the periodicity, and sets $x_1=q(x_0)$. If we set $p=2$ and $m=\infty$, the full AAR gives
 \begin{equation}\label{eq:ARRp2}
 	\mathbf{x}_0,  \mathbf{x}_1=q(\mathbf{x}_0), \mathbf{x}_2=q(\mathbf{x}_1), \mathbf{x}_3\leftarrow AA(2), \mathbf{x}_4=q(\mathbf{x}_3), \mathbf{x}_5\leftarrow AA(4),\cdots
 \end{equation}
 Here, $\mathbf{x}_j\leftarrow$ AA($i$) means that $\mathbf{x}_j$ is obtained by AA($i$). However, if we set $s=t=1$ and $m=\infty$ in Algorithm \ref{saAAm_FPn},  aAA($\infty$)[1]--FP[1] generates the sequence
 \begin{equation}\label{eq:aAAp2}
 	\mathbf{x}_0,  \mathbf{x}_1=q(\mathbf{x}_0), \mathbf{x}_2 \leftarrow AA(1), \mathbf{x}_3=q(\mathbf{x}_2), \mathbf{x}_4 \leftarrow AA(3), \mathbf{x}_5=q(\mathbf{x}_4),\cdots
 \end{equation}
 If we set $s=1$ and $t=2$ and $m=\infty$ in Algorithm \ref{saAAm_FPn}, aAA($\infty$)[1]--FP[2] generates the sequence
 \begin{equation}\label{eq:aAAp3}
 	\mathbf{x}_0,  \mathbf{x}_1=q(\mathbf{x}_0), \mathbf{x}_2=q(\mathbf{x}_1), \mathbf{x}_3\leftarrow AA(2), \mathbf{x}_4=q(\mathbf{x}_3), \mathbf{x}_5=q(\mathbf{x}_4), \mathbf{x}_6 \leftarrow AA(5), \cdots
 \end{equation}
 From \eqref{eq:ARRp2}, \eqref{eq:aAAp2}, and \eqref{eq:aAAp3}, it is evident that AAR and aAA($m$)[1]--FP[$t$] are different.

 \begin{algorithm}[H]
 	\caption{aAA($m$)[$s$]--FP[$t$]}\label{saAAm_FPn}
 	\begin{algorithmic}[1]
 		\State{Given integers $m,s,t$, and initial guess $\mathbf{x}_0$}
 		\State{Set $\mathbf{x}_1=q(\mathbf{x}_0)$, $\mathbf{r}_0 = \mathbf{x}_1 - \mathbf{x}_0$, $\mathbf{Q}_k=[\mathbf{x}_1 \, ]$, $\mathbf{R}_k=[\mathbf{r}_0 \, ]$}
 		\For{$k=2$ \textbf{until} convergence,}
 		\State{$\mathbf{\bar{x}}_k=q(\mathbf{x}_{k-1})$}
 		\State{$\mathbf{r}_k= \mathbf{\bar{x}}_k- \mathbf{x}_{k-1}$}
 		\State{$m_k=\min(k-1,m)$}
 		\State{$\mathbf{R}_k=[\mathbf{r}_k, \mathbf{R}_k(:,1:m_k)]$}
 		\State{$\mathbf{Q}_k=[\mathbf{\bar{x}}_k, \mathbf{Q}_k(:,1:m_k)]$}
 		\State{$\mathbf{e}_k=\mathrm{ones}(1,m_k)$}
 		\State{$\mathbf{B}_k=\mathrm{kron}(\mathbf{e}_k, \mathbf{r}_k)-\mathbf{R}_k(:,2:\mathrm{end})$}
 		\State{$\mathbf{C}_k=\mathrm{kron}(\mathbf{e}_k, \mathbf{\bar{x}}_k)-\mathbf{Q}_k(:,2:\mathrm{end})$}
 		\If{$\mod{(k-1,s+t)} < t$}
 		\State{Set $\mathbf{x}_{k}= \mathbf{\bar{x}}_{k}$}
 		\Else
 		\State{Solve ${\bm{\gamma}^{(k)}}=\mathrm{argmin}_{\gamma} \| \mathbf{r}_k + \mathbf{B}_k \mathbf{\gamma}\|$}
 		\State{Set $\mathbf{x}_{k} = \mathbf{\bar{x}}_{k} + \mathbf{C}_k {\bm{\gamma}^{(k)}}$}
 		\EndIf
 		\EndFor
 	\end{algorithmic}
 \end{algorithm}
 The diagrams below illustrate our proposed aAA($m$)[$s$]--FP[$t$] method for various choices of $m, s, t$, to aid understanding. In Figure \ref{aAA(3)[1]--FP[3]} we report the diagram for aAA(3)[1]--FP[3]; in this case, we perform three fixed point iterations followed by one AA(3). Further, in Figure \ref{aAA(2)[2]--FP[1]} we report the diagram for aAA(2)[2]--FP[1]; within this setting, we perform one fixed point iteration followed by two AA(2); note that the first Anderson loop performs the full Anderson, that is, the algorithm performs AA(1) followed by AA(2); after this, enough information on the residual is available; thus the algorithm performs AA(2) afterwards.   Figures \ref{aAA(3)[3]--FP[5]} and \ref{aAA(infty)[1]--FP[2]} present aAA(3)[3]--FP[5] and aAA($\infty$)[1]--FP[2], respectively.
 
 \begin{figure}[!ht]
 	\begin{tikzpicture}[
 		roundnode/.style={circle, draw=green!60, fill=green!5, thick, minimum size=3.mm},
 		roundnodeb/.style={circle, draw=blue!60, fill=blue!5, thick, minimum size=3.mm},
 		squarednode/.style={rectangle, draw=red!60, fill=red!5, thick, minimum size=3.mm},
 		squarednodeb/.style={rectangle, draw=white!60, fill=white!5, thick, minimum size=3.mm},
 		]
 		\node[roundnodeb]        (x0)                          {$\mathbf{x}_0$};
 		\node[roundnode]        (fp1)       [right=3mm of x0]  {FP};
 		\node[roundnode]        (fp2)       [right=3mm of fp1] {FP};
 		\node[roundnode]        (fp3)       [right=3mm of fp2] {FP};
 		\node[squarednode]      (AA1)       [right=3mm of fp3] {AA(3)};
 		\node[roundnode]        (fp4)       [right=3mm of AA1] {FP};
 		\node[roundnode]        (fp5)       [right=3mm of fp4] {FP};
 		\node[roundnode]        (fp6)       [right=3mm of fp5] {FP};
 		\node[squarednode]      (AA2)       [right=3mm of fp6] {AA(3)};
 		\node[squarednodeb]     (end)       [right=3mm of AA2] {$\cdots$};
 		
 		\draw[->] (x0.east) -- (fp1.west);
 		\draw[->] (fp1.east) -- (fp2.west);
 		\draw[->] (fp2.east) -- (fp3.west);
 		\draw[->] (fp3.east) -- (AA1.west);
 		\draw[->] (AA1.east) -- (fp4.west);
 		\draw[->] (fp4.east) -- (fp5.west);
 		\draw[->] (fp5.east) -- (fp6.west);
 		\draw[->] (fp6.east) -- (AA2.west);
 		\draw[->] (AA2.east) -- (end.west);
 	\end{tikzpicture}
 	\caption{Diagram loop of aAA(3)[1]--FP[3]}\label{aAA(3)[1]--FP[3]}
 \end{figure}

 \begin{figure}[!ht]
 	\begin{tikzpicture}[
 		roundnode/.style={circle, draw=green!60, fill=green!5, thick, minimum size=4mm},
 		roundnodeb/.style={circle, draw=blue!60, fill=blue!5, thick, minimum size=4mm},
 		squarednode/.style={rectangle, draw=red!60, fill=red!5, thick, minimum size=4mm},
 		squarednodeb/.style={rectangle, draw=white!60, fill=white!5, thick, minimum size=4mm},
 		]
 		\node[roundnodeb]        (x0)                          {$\mathbf{x}_0$};
 		\node[roundnode]        (fp1)       [right=3mm of x0]  {FP};
 		\node[squarednode]      (AA1)       [right=3mm of fp1] {AA(1)};
 		\node[squarednode]      (AA2)       [right=3mm of AA1] {AA(2)};
 		\node[roundnode]        (fp2)       [right=3mm of AA2] {FP};
 		\node[squarednode]      (AA3)       [right=3mm of fp2] {AA(2)};
 		\node[squarednode]      (AA4)       [right=3mm of AA3] {AA(2)};
 		\node[roundnode]        (fp3)       [right=3mm of AA4] {FP};
 		\node[squarednodeb]     (end)       [right=3mm of fp3] {$\cdots$};
 		
 		\draw[->] (x0.east) -- (fp1.west);
 		\draw[->] (fp1.east) -- (AA1.west);
 		\draw[->] (AA1.east) -- (AA2.west);
 		\draw[->] (AA2.east) -- (fp2.west);
 		\draw[->] (fp2.east) -- (AA3.west);
 		\draw[->] (AA3.east) -- (AA4.west);
 		\draw[->] (AA4.east) -- (fp3.west);
 		\draw[->] (fp3.east) -- (end.west);
 	\end{tikzpicture}
 	\caption{Diagram loop of aAA(2)[2]--FP[1]}\label{aAA(2)[2]--FP[1]}
 \end{figure}

 \begin{figure}[!ht]
 	\begin{tikzpicture}[
 		roundnode/.style={circle, draw=green!60, fill=green!5, thick, minimum size=4mm},
 		roundnodeb/.style={circle, draw=blue!60, fill=blue!5, thick, minimum size=4mm},
 		squarednode/.style={rectangle, draw=red!60, fill=red!5, thick, minimum size=4mm},
 		squarednodeb/.style={rectangle, draw=white!60, fill=white!5, thick, minimum size=4mm},
 		]
 		\node[roundnodeb]        (x0)                          {$\mathbf{x}_0$};
 		\node[roundnode]        (fp1)       [right=3mm of x0]  {FP};
 		\node[roundnode]        (fp2)       [right=3mm of fp1] {FP};
 		\node[roundnode]        (fp3)       [right=3mm of fp2] {FP};
 		\node[roundnode]        (fp4)       [right=3mm of fp3] {FP};
 		\node[roundnode]        (fp5)       [right=3mm of fp4] {FP};
 		\node[squarednode]      (AA1)       [right=3mm of fp5] {AA(3)};
 		\node[squarednode]      (AA2)       [right=3mm of AA1] {AA(3)};
 		\node[squarednode]      (AA3)       [right=3mm of AA2] {AA(3)};
 		\node[squarednodeb]     (end)       [right=3mm of AA3] {$\cdots$};
 		
 		\draw[->] (x0.east) -- (fp1.west);
 		\draw[->] (fp1.east) -- (fp2.west);
 		\draw[->] (fp2.east) -- (fp3.west);
 		\draw[->] (fp3.east) -- (fp4.west);
 		\draw[->] (fp4.east) -- (fp5.west);
 		\draw[->] (fp5.east) -- (AA1.west);
 		\draw[->] (AA1.east) -- (AA2.west);
 		\draw[->] (AA2.east) -- (AA3.west);
 		\draw[->] (AA3.east) -- (end.west);
 	\end{tikzpicture}
 	\caption{Diagram loop of aAA(3)[3]--FP[5]}\label{aAA(3)[3]--FP[5]}
 \end{figure}

 \begin{figure}[!ht]
 	\begin{tikzpicture}[
 		roundnode/.style={circle, draw=green!60, fill=green!5, thick, minimum size=4mm},
 		roundnodeb/.style={circle, draw=blue!60, fill=blue!5, thick, minimum size=4mm},
 		squarednode/.style={rectangle, draw=red!60, fill=red!5, thick, minimum size=4mm},
 		squarednodeb/.style={rectangle, draw=white!60, fill=white!5, thick, minimum size=4mm},
 		]
 		\node[roundnodeb]        (x0)                          {$\mathbf{x}_0$};
 		\node[roundnode]        (fp1)       [right=3mm of x0]  {FP};
 		\node[roundnode]        (fp2)       [right=3mm of fp1] {FP};
 		\node[squarednode]      (AA1)       [right=3mm of fp2] {AA(2)};
 		\node[roundnode]        (fp3)       [right=3mm of AA1] {FP};
 		\node[roundnode]        (fp4)       [right=3mm of fp3] {FP};
 		\node[squarednode]      (AA2)       [right=3mm of fp4] {AA(5)};
 		\node[roundnode]        (fp5)       [right=3mm of AA2] {FP};
 		\node[roundnode]        (fp6)       [right=3mm of fp5] {FP};
 		\node[squarednode]      (AA3)       [right=3mm of fp6] {AA(8)};
 		\node[squarednodeb]     (end)       [right=3mm of AA3] {$\cdots$};
 		
 		\draw[->] (x0.east) -- (fp1.west);
 		\draw[->] (fp1.east) -- (fp2.west);
 		\draw[->] (fp2.east) -- (AA1.west);
 		\draw[->] (AA1.east) -- (fp3.west);
 		\draw[->] (fp3.east) -- (fp4.west);
 		\draw[->] (fp4.east) -- (AA2.west);
 		\draw[->] (AA2.east) -- (fp5.west);
 		\draw[->] (fp5.east) -- (fp6.west);
 		\draw[->] (fp6.east) -- (AA3.west);
 		\draw[->] (AA3.east) -- (end.west);
 	\end{tikzpicture}
 	\caption{Diagram loop of aAA($\infty$)[1]--FP[2]}\label{aAA(infty)[1]--FP[2]}
 \end{figure}

 Our proposed algorithm can be applied to both linear and nonlinear problems. We will provide a convergence analysis for the linear case. 
 
 \section{Convergence analysis for the linear case}\label{sec:convergence}
 Consider solving
 \begin{equation}\label{eq:Ax=b}
 	\mathbf{Ax}=\mathbf{b}, 
 \end{equation}
 where $\mathbf{A}\in\mathbb{R}^{n\times n}$ and $\mathbf{x}^*\in\mathbb{R}^{n}$ is the exact solution. We consider the Richardson iteration as the fixed-point iteration for solving \eqref{eq:Ax=b}, which is given by
 \begin{equation}\label{eq:FP-linear}
 	\mathbf{x}_{k+1}=q(\mathbf{x}_k)=\mathbf{M}\mathbf{x}_k+\mathbf{b},
 \end{equation}
 where $\mathbf{M}=\mathbf{I}-\mathbf{A}$. We define the $k$th error as $\mathbf{e}_k=\mathbf{x}_k-\mathbf{x}^*$. From \eqref{eq:kthresidual}, we have $\mathbf{r}_k=\mathbf{b}-\mathbf{A}\mathbf{x}_k=-\mathbf{A}\mathbf{e}_k$. In this section, we provide theoretical results on the convergence behavior of aAA($m$)[$s$]--FP[$t$] for some special choices of $m, s, t$ and $\mathbf{M}$. 
 
 In the literature, it has been shown \cite{walker2011anderson,potra2013characterization} that under certain conditions the iterates, $\mathbf{x}_k$, of AA($\infty$) can be recovered from GMRES iterate $\mathbf{x}_{k-1}^G$, i.e., $\mathbf{x}_k=q(\mathbf{x}_{k-1}^G)$. In \cite{lupo2019convergence}, the author establishes a connection between full AAR and GMRES.  We present a similar result for aAA($\infty$)[$1$]--FP[$t$], which can be derived following \cite[Theorem 1]{lupo2019convergence}.  
 \begin{theorem} 
 	Consider the aAA($\infty$)[$1$]--FP[$t$] method described in Algorithm \ref{saAAm_FPn}, applied to solve \eqref{eq:Ax=b} using the fixed-point iteration \eqref{eq:FP-linear}. Denote its iterates as $\mathbf{x}_k$. Let $p=t+1$. Assume that there exists a $k_0$ such that the norms of the residuals of GMRES are strictly decreasing for the first $k_0$th iteration. Then, we have 
 	\begin{equation}\label{eq:eqv-p-period}
 		\mathbf{x}_{jp} =q(\mathbf{x}_{jp-1}^G), \quad  jp< k_0,
 	\end{equation}
 	where $\mathbf{x}_{jp-1}^G$ is the $(jp-1)$th iterate of GMRES.
 \end{theorem}
 \begin{proof}
 	Let $jp=k+1$. From \cite{de2024anderson}, we know that the $(k+1)$th update of aAA($\infty$)[1]--FP[$t$] (i.e., $m_k=k$) is given by
 	\begin{equation*}
 		\mathbf{x}_{k+1}=(1+\sum_{i=1}^k\gamma_i)\mathbf{M}\mathbf{x}_k-\sum_{i=1}^k\gamma_iM\mathbf{x}_{k-i}+\mathbf{b}=\mathbf{M}( (1+\sum_{i=1}^k\gamma_i)\mathbf{x}_k-\sum_{i=1}^k\gamma_i\mathbf{x}_{k-i})+\mathbf{b}=q(\mathbf{\hat{x}}_k),
 	\end{equation*}
 	where $\mathbf{\hat{x}}_k=(1+\sum_{i=1}^k\gamma_i)\mathbf{x}_k-\sum_{i=1}^k\gamma_i\mathbf{x}_{k-i}$. The minimizing problem \eqref{eq:min-AA} is
 	\begin{equation*}
 		\min_{\left(\gamma_1^{(k)},\cdots, \gamma_{k}^{(k)}\right)} \left\|(1+\sum_{i=1}^k\gamma_i)\mathbf{r}_k-\sum_{i=1}^k\gamma_i\mathbf{r}_{k-i}\right\|^2=\min_{\left(\gamma_1^{(k)},\cdots, \gamma_{k}^{(k)}\right)} \|r(\mathbf{\hat{x}}_k)\|^2.
 	\end{equation*}
 	Following the idea in \cite[Theorem 1]{lupo2019convergence}, we have $\mathbf{\hat{x}}_k=\mathbf{x}_k^G$, where $k=jp-1$. Thus, we obtain the desired result.
 \end{proof} 
 
 Note that in the above theorem, when $t=0$, aAA($\infty$)[$1$]--FP[$t$] is AA($\infty$), and \eqref{eq:eqv-p-period} is $ \mathbf{x}_{j} =q(\mathbf{x}_{j-1}^G)$, which is consistent with the well-known result in \cite{walker2011anderson,potra2013characterization}. 
 
 For the AA($m$) method, \cite{Greif_He} has shown that the least-squares problem defined in \eqref{eq:min-AA} is minimizing $\|\mathbf{M}^{-1}\mathbf{r}_{k+1}\|$, provided that $\mathbf{M}$ is invertible, that is,
 \begin{displaymath}\label{eq:AA-minrk}
 \min_{\left(\gamma_1^{(k)},\cdots, \gamma_{k}^{(k)}\right) }\|\mathbf{M}^{-1}\mathbf{r}_{k+1}\|=	\min_{\left(\gamma_1^{(k)},\cdots, \gamma_{k}^{(k)}\right)}\|r(\mathbf{x}_k)+\sum_{i=1}^{m_k} \gamma^{(k)}_i\left(r(\mathbf{x}_k)-r(\mathbf{x}_{k-i})\right)\|^2.
 \end{displaymath}
Moreover, \cite{Greif_He} has established that
 \begin{equation}\label{eq:AA-nonincre}
 	\|\mathbf{M}^{-1}\mathbf{r}_{k+1}\|\leq \|r_k\|.
 \end{equation}
 Although \eqref{eq:AA-nonincre} does not guarantee that $\|\mathbf{r}_k\|$ decreases monotonically, we will show that this is indeed the case when $\|\mathbf{M}\|<1$.  Next, we provide a convergence analysis for $\|\mathbf{M}\|<1$ for aAA($m$)[$s$]--FP[$t$].
 
 \begin{theorem}
 	Assume that $\|\mathbf{M}\|=c<1$. Then, for any initial guess $\mathbf{x}_0$, the aAA($m$)[$s$]--FP[$t$] method converges.
 \end{theorem}
 \begin{proof}
 	According to the definition of aAA($m$)[$s$]--FP[$t$], we know that its iterate $\mathbf{x}_{k+1}$ is either generated from AA($m$) or the fixed-point iteration. If $\mathbf{x}_{k+1}$ is obtained from AA($m$), then from \cite[Theorem 2.1]{toth2015convergence}, we know $\|\mathbf{r}_{k+1}\|\leq c \|\mathbf{r}_k\|$. If $\mathbf{x}_{k+1}$ is obtained by applying one fixed-point iteration given by \eqref{eq:FP-linear}, then we have $\mathbf{r}_{k+1}=\mathbf{M}\mathbf{r}_k$. It follows that $\|\mathbf{r}_{k+1}\|\leq c \|\mathbf{r}_k\|$. As a result, $\|\mathbf{r}_{k+1}\|\leq c \|\mathbf{r}_k\|\leq c^{k+1}\|\mathbf{r}_0\|$ for all $k$. Since $c<1$, $\|\mathbf{r}_{k+1}\|\rightarrow 0$.
 \end{proof}
 
 For $\|\mathbf{M}\|\geq 1$, it is not clear whether the aAA($m$)[$s$]--FP[$t$] method converges. In the following, we will derive some sufficient conditions such that aAA($m$)[$s$]--FP[$t$] converges. Before doing that, we define the following matrices with dimensions $n\times m_k$, 
 \begin{equation*}
 	\mathbf{X}_k= 
 	\begin{bmatrix}
 		\mathbf{x}_{k-m_k+1}-\mathbf{x}_{k-m_k},& \mathbf{x}_{k-m_k+2}-\mathbf{x}_{k-m_k+1}, &\cdots, & \mathbf{x}_k-\mathbf{x}_{k-1}
 	\end{bmatrix},
 \end{equation*}
 and 
 \begin{equation*}
 	\mathbf{D}_k=
 	\begin{bmatrix}
 		r(\mathbf{x}_{k-m_k+1})-r(\mathbf{x}_{k-m_k}),& r(\mathbf{x}_{k-m_k+2})-r(\mathbf{x}_{k-m_k+1}),& \cdots, & r(\mathbf{x}_k)-r(\mathbf{x}_{k-1})
 	\end{bmatrix}.
 \end{equation*}
 Let
 \begin{equation*}
 	\bm{\tau}^{(k)}=\begin{bmatrix}
 		\tau_{m_k}^{(k)}, & \tau_{m_k-1}^{(k)},& \cdots, &\tau_{1}^{(k)}
 	\end{bmatrix}^T. 
 \end{equation*}
 Then, the least-squares problem \eqref{eq:AA-LSQ-alt} is equivalent to
 \begin{equation*} 
 	\min_{\bm{\tau}^{(k)}}\|r(\mathbf{x}_k)+\mathbf{D}_k\bm{\tau}^{(k)} \|^2,
 \end{equation*}
 which leads to
 \begin{equation*}
 	\bm{\tau}^{(k)}=- (\mathbf{D}_k^T \mathbf{D}_k)^{-1}\mathbf{D}_k^T r(\mathbf{x}_k).
 \end{equation*}
 We rewrite $\mathbf{X}_k$ and $\mathbf{D}_k$ as  
 \begin{equation}\label{eq:defXk-ek}
 	\mathbf{X}_k=
 	\begin{bmatrix}
 		\mathbf{e}_{k-m_k+1}-\mathbf{e}_{k-m_k}& 	\mathbf{e}_{k-m_k+2}-\mathbf{e}_{k-m_k+1}, & \cdots,& \mathbf{e}_k-\mathbf{e}_{k-1}
 	\end{bmatrix}.
 \end{equation}
 and 
 \begin{equation*}
 	\mathbf{D}_k=-\mathbf{A}
 	\begin{bmatrix}
 		\mathbf{e}_{k-m_k+1}-\mathbf{e}_{k-m_k}& 	\mathbf{e}_{k-m_k+2}-\mathbf{e}_{k-m_k+1}, & \cdots,& \mathbf{e}_k-\mathbf{e}_{k-1}
 	\end{bmatrix}=-\mathbf{A}\mathbf{X}_k.
 \end{equation*}
 
 Using the above notations,  \cite{yang2022anderson}  has shown that \eqref{AA-alt-update} can be rewritten as
 \begin{equation*} 
 	\mathbf{x}_{k+1} =\mathbf{x}_k+(\mathbf{I}-\mathbf{S}_k)\mathbf{r}_k,
 \end{equation*} 
 where
 \begin{equation}\label{eq:def-Sk}
 	\mathbf{S}_k =(\mathbf{D}_k+\mathbf{X}_k)(\mathbf{D}_k^T\mathbf{D}_k)^{-1}\mathbf{D}_k^T.
 \end{equation}
 In addition,
 \begin{equation}\label{eq:ekrk}
 	\mathbf{e}_{k+1}= (\mathbf{I}-(\mathbf{I}-\mathbf{S}_k)\mathbf{A})\mathbf{e}_k, \quad  \mathbf{r}_{k+1}= (\mathbf{I}-\mathbf{A}(\mathbf{I}-\mathbf{S}_k))\mathbf{r}_k.
 \end{equation}
 For our analysis, we define the following Krylov matrix
 \begin{displaymath}
 	\mathcal{K}_s(\mathbf{A},\mathbf{b})=[\mathbf{b}, \mathbf{Ab}, \cdots, \mathbf{A}^{s-1}\mathbf{b}] \in \mathbb{R}^{n\times s}.
 \end{displaymath}

 In \cite{yang2022anderson}, the authors present a one-step analysis of AA when $\mathbf{M}$ is symmetric, that is, comparing the solution error $\mathbf{e}_{k+1}$ after one step of AA, where $\{\mathbf{x}_j=q(\mathbf{x}_{j-1})\}_{j=1}^k$, to the solution error after $(k+1)$ fixed-point iterations.  This idea can be adapted to perform the convergence analysis of our aAA($m$)[$s$]--FP[$t$] method. We first extend the results of \cite[Lemma 3.1]{yang2022anderson}, originally formulated for the symmetric case, to the more general scenario, where $\mathbf{M}$ is diagonalizable and invertible. 
 \begin{lemma}\label{lem:diag-ek-rk}
 	Given $\mathbf{x}_0$, let $\{\mathbf{x}_j=q(\mathbf{x}_{j-1})\}_{j=1}^k$. Assume that $\mathbf{x}_{k+1}$ is obtained by AA($m$) with $m\leq k$ defined in \eqref{eq:xkp1-AA}. Let $\mathbf{M}$ be an $n\times n$  diagonalizable matrix with eigenvalue decomposition $\mathbf{M}= \mathbf{W}\mathbf{\Lambda} \mathbf{W}^{-1}$. Suppose that the eigenvalues of $\mathbf{M}$ are contained in an interval $[a, b]$ that does not contain 0 or 1. Then,
 	\begin{equation}\label{eq:ekplus-ek-relation}
 		\mathbf{e}_{k+1}=\mathbf{W} \mathbf{\tilde{E}}\mathbf{W}^{-1}\mathbf{M}\mathbf{e}_k,
 	\end{equation}
 	and
 	\begin{displaymath}
 		\mathbf{r}_{k+1}=\mathbf{M} \mathbf{R}\mathbf{r}_k.
 	\end{displaymath}
 	Here, $\mathbf{\tilde{E}}=\mathbf{\Pi} (\mathbf{I}-\mathbf{K}(\mathbf{K}^T\mathbf{W}^T\mathbf{W}\mathbf{K})^{-1}\mathbf{K}^T\mathbf{W}^T\mathbf{W})\mathbf{\Pi}^{-1}$  with $\mathbf{\Pi}=\mathbf{\Lambda}(\mathbf{\Lambda}-\mathbf{I})^{-1}$ and $\mathbf{K}=\mathcal{K}_{m_k}(\mathbf{\Lambda},\mathbf{H}\mathbf{W}^{-1}\mathbf{e}_0)$, where $\mathbf{H}=(\mathbf{\Lambda}-\mathbf{I})^2\mathbf{\Lambda}^{k-m_k}$, and 
 	\begin{equation}\label{eq:def-R-lem1}        
 		\mathbf{R}=\mathbf{I}-\mathbf{W}\mathbf{K}(\mathbf{K}^T\mathbf{W}^T\mathbf{W}\mathbf{K})^{-1}\mathbf{K}^T\mathbf{W}^T.
 	\end{equation}
 \end{lemma}
 \begin{proof}
 We first show that
 	\begin{equation*}
 		\mathbf{X}_k=\mathbf{W}(\mathbf{\Lambda}-\mathbf{I})^{-1}\mathbf{K}\quad \text{and} \quad \mathbf{D}_k=\mathbf{WK}.
 	\end{equation*}
 For $1\leq j\leq k$, $\mathbf{x}_j=\mathbf{M}\mathbf{x}_{j-1}+b$, we have $\mathbf{e}_j=\mathbf{M}^j\mathbf{e}_0$. It follows that $$\mathbf{e}_j-\mathbf{e}_{j-1}=(\mathbf{W\Lambda}^j\mathbf{W}^{-1}-\mathbf{W\Lambda}^{j-1}\mathbf{W}^{-1})\mathbf{e}_0=\mathbf{W(\Lambda-I)\Lambda}^{j-1}\mathbf{W}^{-1}\mathbf{e}_0.$$
 		From \eqref{eq:defXk-ek}, we have
 		\begin{align*}
 			\mathbf{X}_k&=\mathbf{W(\Lambda-I)}\mathbf{\Lambda}^{k-m_k}[\mathbf{W}^{-1}\mathbf{e}_0, \mathbf{\Lambda W}^{-1}\mathbf{e}_0, \cdots, \mathbf{\Lambda}^{m_k-1} \mathbf{W}^{-1}\mathbf{e}_0]\\
 			&=\mathbf{W(\Lambda-I)}^{-1} (\mathbf{\Lambda-I})^{2}\mathbf{\Lambda}^{k-m_k}[\mathbf{W}^{-1}\mathbf{e}_0, \mathbf{\Lambda W}^{-1}\mathbf{e}_0, \cdots, \mathbf{\Lambda}^{m_k-1} \mathbf{W}^{-1}\mathbf{e}_0]\\
 			&=\mathbf{W(\Lambda-I)}^{-1}[\mathbf{HW}^{-1}\mathbf{e}_0, \mathbf{\Lambda HW}^{-1}\mathbf{e}_0, \cdots, \mathbf{\Lambda}^{m_k-1} \mathbf{HW}^{-1}\mathbf{e}_0]\\
 			&=\mathbf{W(\Lambda-I)}^{-1}\mathcal{K}_{m_k}(\mathbf{\Lambda},\mathbf{H}\mathbf{W}^{-1}\mathbf{e}_0),\\
 			&=\mathbf{W(\Lambda-I)}^{-1}\mathbf{K},
 		\end{align*}
 		which is the desired result.

 	Since $\mathbf{D}_k=-\mathbf{A}\mathbf{X}_k$ and $\mathbf{A}=\mathbf{I}-\mathbf{M}$, we have
 	\begin{equation*}
 		\mathbf{D}_k=-(\mathbf{I}-\mathbf{W\Lambda W}^{-1})\mathbf{W}(\mathbf{\Lambda}-\mathbf{I})^{-1}\mathbf{K}=\mathbf{WK}.
 	\end{equation*}
 	Moreover, using \eqref{eq:def-Sk} leads to
 	\begin{equation*}
 		\mathbf{S}_k=\mathbf{W}\mathbf{\Lambda}(\mathbf{\Lambda}-\mathbf{I})^{-1}\mathbf{K}(\mathbf{K}^T\mathbf{W}^T\mathbf{WK})^{-1}\mathbf{K}^T\mathbf{W}^T.
 	\end{equation*}
 	Then, we can further obtain
 	\begin{align*}
 		&\mathbf{I}-(\mathbf{I}-\mathbf{S}_k)\mathbf{A}\\
 		=&\mathbf{M}+\mathbf{S}_k(\mathbf{I}-\mathbf{M})\\
 		=&\mathbf{W}\mathbf{\Lambda}\mathbf{W}^{-1} + \mathbf{W}\mathbf{\Lambda}(\mathbf{\Lambda}-\mathbf{I})^{-1}\mathbf{K}(\mathbf{K}^T\mathbf{W}^T\mathbf{WK})^{-1}\mathbf{K}^T\mathbf{W}^T\mathbf{W}(\mathbf{I}-\mathbf{\Lambda}) \mathbf{W}^{-1}\\
 		=&\mathbf{W}\mathbf{\Lambda}(\mathbf{\Lambda}-\mathbf{I})^{-1}\left(\mathbf{I}- \mathbf{K}(\mathbf{K}^T\mathbf{W}^T\mathbf{WK})^{-1}\mathbf{K}^T\mathbf{W}^T \mathbf{W}\right) (\mathbf{\Lambda}-\mathbf{I})\mathbf{W}^{-1}\\
 		=&\mathbf{W}\mathbf{\Lambda}(\mathbf{\Lambda}-\mathbf{I})^{-1}\left(\mathbf{I}- \mathbf{K}(\mathbf{K}^T\mathbf{W}^T\mathbf{W}\mathbf{K})^{-1}\mathbf{K}^T\mathbf{W}^T \mathbf{W}\right) (\mathbf{\Lambda}-\mathbf{I}) \mathbf{\Lambda}^{-1}\mathbf{W}^{-1}\mathbf{M}.
 	\end{align*}
 	Using \eqref{eq:ekrk}, we obtain the first desired result.
 	
 	Next, we compute
 	\begin{align*}
 		\mathbf{I}-\mathbf{A}(\mathbf{I}-\mathbf{S}_k)&=\mathbf{M}+(\mathbf{I}-\mathbf{M})\mathbf{S}_k\\
 		&=\mathbf{W\Lambda}\mathbf{W}^{-1} +\mathbf{W}(\mathbf{I}-\mathbf{\Lambda})\mathbf{\Lambda}(\mathbf{\Lambda}-\mathbf{I})^{-1}\mathbf{K}(\mathbf{K}^T\mathbf{W}^T\mathbf{WK})^{-1}\mathbf{K}^T\mathbf{W}^T \\
 		&=\mathbf{W\Lambda} \mathbf{W}^{-1} - \mathbf{W}\mathbf{\Lambda K}(\mathbf{K}^T\mathbf{W}^T\mathbf{WK})^{-1}\mathbf{K}^T\mathbf{W}^T \\
 		&=\mathbf{W}\mathbf{\Lambda} \mathbf{W}^{-1}(\mathbf{I}-\mathbf{WK}(\mathbf{K}^T\mathbf{W}^T\mathbf{WK})^{-1}\mathbf{K}^T\mathbf{W}^T)\\
 		&=\mathbf{M}(\mathbf{I}-\mathbf{WK}(\mathbf{K}^T\mathbf{W}^T\mathbf{WK})^{-1}\mathbf{K}^T\mathbf{W}^T).
 	\end{align*}
 	Using \eqref{eq:ekrk}, we obtain the second desired result.  
 \end{proof}
 We note that when $\mathbf{M}$ is symmetric, $\mathbf{W}$ is a unitary matrix, that is, $\mathbf{W}^T\mathbf{W}=\mathbf{I}$. Consequently, the relationship \eqref{eq:ekplus-ek-relation} coincides with the result given in \cite[Lemma 3.1]{yang2022anderson}. Next, we provide an estimate for $\|\mathbf{r}_{k+1}\|$ based on the above result.
 
 \begin{theorem}\label{thm:one-step-gain}
 	Assume that the assumptions in Lemma \ref{lem:diag-ek-rk} hold. Then,
 	\begin{equation}\label{eq:one-AAm}
 		\|\mathbf{r}_{k+1}\|\leq C(a,b,m)\kappa_2(\mathbf{W})\|\mathbf{M}\|\|\mathbf{r}_k\|,
 	\end{equation}
 	or
 	\begin{equation}\label{eq:gain-AAm}
 		\|\mathbf{r}_{k+1}\|\leq C(a,b,m)\kappa_2(\mathbf{W})\|\mathbf{M}\|^{k+1}\|\mathbf{r}_0\|.
 	\end{equation}
 	where $\kappa_2(\mathbf{W})=\|\mathbf{W}\|\|\mathbf{W}^{-1}\|$,  and
 	\begin{equation*}
 		C(a,b,m)=\left|T_{m}\left(\frac{2ab-a-b}{b-a}\right)\right|^{-1},
 	\end{equation*}
 	where $T_{m}(x)$ is the Chebyshev polynomial of degree $m$.
 \end{theorem} 
 \begin{proof}
 	From Lemma \ref{lem:diag-ek-rk}, we have $\|\mathbf{r}_{k+1}\|\leq \|\mathbf{M}\|\|\mathbf{R}\mathbf{r}_k\|$. We then estimate $\|\mathbf{R}\mathbf{r}_k\|$. Using the definition of $\mathbf{R}$ in \eqref{eq:def-R-lem1}, we have	
 	\begin{align*}
 		\|\mathbf{R}\mathbf{r}_k\|&=\|(\mathbf{I}-\mathbf{WK}(\mathbf{K}^T\mathbf{W}^T\mathbf{WK})^{-1}\mathbf{K}^T\mathbf{W}^T)\mathbf{r}_k\|\\
 		&=\min_{\mathbf{c}\in \mathbb{R}^{m_k}}\|\mathbf{r}_k-\mathbf{WK}\mathbf{c}\|\\
 		&=\min_{p_{\ell}\in \mathbb{P}_{m_k-1}}\| \mathbf{r}_k-p_{\ell}(\mathbf{M})\mathbf{WH}\mathbf{W}^{-1}\mathbf{e}_0\|.
 	\end{align*} 
 	However, for $1\leq j\leq k, \mathbf{x}_j=q(\mathbf{x}_{j-1})=\mathbf{M}\mathbf{x}_{j-1}+b$. It follows that $\mathbf{e}_j=\mathbf{M}\mathbf{x}_{j-1}+\mathbf{b}-\mathbf{x}^*=\mathbf{M}\mathbf{e}_{j-1}=\mathbf{M}^j\mathbf{e}_{0}$. Then,
 	$\mathbf{W}^{-1}\mathbf{M}\mathbf{e}_k=\mathbf{W}^{-1}\mathbf{M}^{k+1}\mathbf{e}_0=\mathbf{\Lambda}^{k+1}\mathbf{W}^{-1}\mathbf{e}_0$. Thus,
 	\begin{align*}
 		\mathbf{WH}\mathbf{W}^{-1}\mathbf{e}_0&= \mathbf{W}(\mathbf{\Lambda}-\mathbf{I})^2\mathbf{\Lambda}^{k-m_k}\mathbf{W}^{-1}\mathbf{e}_0\\
 		&=\mathbf{W}(\mathbf{I}-\mathbf{\Lambda})\mathbf{\Lambda}^{-m_k-1}(\mathbf{I}-\mathbf{\Lambda})\mathbf{\Lambda}^{k+1}\mathbf{W}^{-1}\mathbf{e}_0\\
 		&= \mathbf{W}(\mathbf{I}-\mathbf{\Lambda})\mathbf{\Lambda}^{-m_k-1}(\mathbf{I}-\mathbf{\Lambda}) \mathbf{W}^{-1}\mathbf{M}\mathbf{e}_k\\
 		&= \mathbf{W}(\mathbf{I}-\mathbf{\Lambda})\mathbf{\Lambda}^{-m_k-1}\mathbf{W}^{-1}\mathbf{W}(\mathbf{I}-\mathbf{\Lambda}) \mathbf{W}^{-1}\mathbf{M}\mathbf{e}_k\\
 		&= -\mathbf{W}(\mathbf{I}-\mathbf{\Lambda})\mathbf{\Lambda}^{-m_k-1}\mathbf{W}^{-1}\mathbf{M}\mathbf{r}_k\\
 		&= -\mathbf{W}(\mathbf{I}-\mathbf{\Lambda})\mathbf{\Lambda}^{-m_k}\mathbf{W}^{-1}\mathbf{r}_k,
 	\end{align*}  
 where in the second-to-last equality we use the relation $\mathbf{r}_k=-\mathbf{A}\mathbf{e}_k=-(\mathbf{I-M})\mathbf{e}_k$.
 	
 	We can further simplify  
 	\begin{align*}	        
 		\|\mathbf{R}\mathbf{r}_k\| &=\min_{p_{\ell}\in \mathbb{P}_{m_k-1}}\left\| \left(\mathbf{I}-p_{\ell}(\mathbf{W\Lambda}\mathbf{W}^{-1})\mathbf{W}(\mathbf{I}-\mathbf{\Lambda})\mathbf{\Lambda}^{-m_k}\mathbf{W}^{-1}\right)\mathbf{r}_k\right\|\\
 		&\leq \|\mathbf{W}\|\|\mathbf{W}^{-1}\|\min_{p_{\ell}\in \mathbb{P}_{m_k-1}}\left\| \left(\mathbf{I}-p_s(\mathbf{\Lambda}) (\mathbf{I}-\mathbf{\Lambda})\mathbf{\Lambda}^{-m_k}\right)\right\|\|\mathbf{r}_k\|\\
 		& \leq \kappa_2(\mathbf{W})\min_{p_{\ell}\in P_{m_k-1}} \max_{t\in [a,b]}|1-p_{\ell}(t) t^{-m_k}(1-t)| \|\mathbf{r}_k\|\\
 		& \leq \kappa_2(\mathbf{W})\min_{p_{\ell} \in P_{m_k-1}} \max_{t\in [1/b,1/a]}|1-p_{\ell}(1/t) t^{m_k-1}(t-1)| \|\mathbf{r}_k\|\\
 		& \leq \kappa_2(\mathbf{W})\min_{p_{\ell} \in P_{m_k-1}} \max_{t\in [1/b,1/a]}|1-p_{\ell}(t)(t-1)| \|\mathbf{r}_k\|\\
 		& \leq \kappa_2(\mathbf{W})\min_{q\in P_{m_k}, q(1)=1} \max_{t\in [1/b,1/a]}|q(t)| \|\mathbf{r}_k\|\\
 		&\leq \kappa_2(\mathbf{W}) C(a,b,m_k)\| \mathbf{r}_k\|,
 	\end{align*}
 	where the last inequality can be found in the proof of \cite[Theorem 3.1]{yang2022anderson}. Because $m\leq k$, $m_k=\min\{m,k\}=m$ and we have the desired result \eqref{eq:one-AAm}. Recall that $\mathbf{x}_k$ is obtained by applying $k$ times fixed-point iteration to $\mathbf{x}_0$. Thus, $\mathbf{r}_k=\mathbf{M}^k\mathbf{r}_0$. It follows that \eqref{eq:gain-AAm} holds.
 \end{proof} 
 If $\|\mathbf{M}\|\leq 1$, the term $C(a,b,m)\kappa_2(\mathbf{W})$ in \eqref{eq:one-AAm} is the gain obtained by one step AA($m$) compared with applying $(k+1)$ fixed-point iterations to $\mathbf{x}_0$. If $\|\mathbf{M}\|>1$, the term $C(a,b,m)\kappa_2(\mathbf{W})\|\mathbf{M}\|^{k+1}$ in \eqref{eq:gain-AAm} could be less than one, which will help us to build the convergence for the subsequence $\{\mathbf{r}_{jp}\}$ obtained from aAA($m$)[$s$]--FP[$t$]. We estimate for $ \|\mathbf{r}_{jp}\|$ in the following.
 
 \begin{theorem}\label{thm:aAA-convergence}
 	Let $\mathbf{M}$ be an $n\times n$  diagonalizable matrix with eigenvalue decomposition $\mathbf{M}= \mathbf{W\Lambda}\mathbf{W}^{-1}$. Suppose that the eigenvalues of $\mathbf{M}$ are contained in an interval $[a, b]$ that does not contain 0 or 1.  Consider aAA($m$)[$s$]--FP[$t$] with $m\leq t$ in Algorithm \ref{saAAm_FPn}. Let $p=t+s$.  Then, we have 
 	\begin{equation}\label{eq:aAAsFPt}
 		\|\mathbf{r}_{jp}\|\leq   (C(a,b,m)\kappa_2(\mathbf{W}))^j\|\mathbf{M}\|^{ jp}\|\mathbf{r}_0\|.
 	\end{equation}
 \end{theorem}
 \begin{proof}
 	Note that the $s$ iterates $\{\mathbf{x}_{jp-d}\}_{d=0}^{s-1}$ of aAA($m$)[$s$]--FP[$t$] are obtained by AA steps.  
 	From \cite[Theorem 2.1]{toth2015convergence}, we have
 	\begin{equation}\label{eq:aAAms}
 		\|\mathbf{r}_{jp}\|\leq \|\mathbf{M}\|\|\mathbf{r}_{jp-1}\|\leq \cdots\leq \|\mathbf{M}\|^{s-1}\|\mathbf{r}_{jp-s+1}\|.
 	\end{equation}
 Note that the $t$ iterates $\{\mathbf{x}_{jp-d}\}_{d=s}^{p-1}$ of aAA($m$)[$s$]--FP[$t$] are generated by the fixed-point iterations, while $\mathbf{x}_{jp-s+1}$ is produced by the AA($m$) step, which uses the previous $m+1$ initial guesses. Since $m\leq t$ corresponding to the condition $m\leq k$ in Lemma \ref{lem:diag-ek-rk} required by Theorem \ref{thm:one-step-gain},   we can apply Theorem \ref{thm:one-step-gain} to the $(jp-s+1)$th iteration. Then, we have 
 	\begin{equation*}
 		\|\mathbf{r}_{jp-s+1}\|\leq  C(a,b,m) \kappa_2(\mathbf{W})\|\mathbf{M}\|\|\mathbf{r}_{jp-s}\|\leq  C(a,b,m) \kappa_2(\mathbf{W})\|\mathbf{M}\|^{p-s+1}\|\mathbf{r}_{(j-1)p}\|,
 	\end{equation*}
 where the first inequality follows from \eqref{eq:one-AAm}, and the second inequality uses the standard property of the error estimation for the fixed-point iterations.
 	
 	Combing the above two results, we obtain
 	\begin{equation}\label{eq:rjp-rjm1p}
 		\|\mathbf{r}_{jp}\| \leq  C(a,b,m)\kappa_2(\mathbf{W})\|\mathbf{M}\|^p\|\mathbf{r}_{(j-1)p}\|
 		\leq ( C(a,b,m)\kappa_2(\mathbf{W}))^j\|\mathbf{M}\|^{jp}\|\mathbf{r}_0\|,
 	\end{equation}
 	which is the desired result.
 \end{proof} 
When $m>t$, the results in Theorem 3 no longer apply. Consequently, we cannot derive the results in Theorem 4. In fact, when $m\leq t$, the $\mathbf{x}_{jp-s+1}$ is a linear combination of $m+1$ fixed-point iterates. While, when $m>t$, the $\mathbf{x}_{jp-s+1}$ is a linear combination of $t$ fixed-point iterates and previous $m+1-t$ iterates generated by AA($m$), so there is no pattern available to estimate $\mathbf{x}_{jp-s+1}$. 
 
 We remark that, in the above theorem, we require $m\leq t$. However, in practice, one can take $m>t$, which we will consider in our numerical tests.  Theorem \ref{thm:aAA-convergence} indicates that the term $(C(a,b,m)\kappa_2(\mathbf{W}))^j$ is the gain obtained by aAA($m$)[$s$]--FP[$t$]  compared with conducting $(jp)$ fixed-point iterations to $\mathbf{x}_0$. If $\|\mathbf{M}\|>1$, the fixed-point iteration is divergent. However, there is a chance that the subsequence $\{\mathbf{r}_{jp}\}$ of aAA($m$)[$s$]--FP[$t$] can converge. In fact, if $C(a,b,m)\kappa_2(\mathbf{W})\|\mathbf{M}\|^p<1$,  the subsequence $\{\mathbf{x}_{jp}\}$ generated by aAA($m$)[$s$]--FP[$t$] is convergent. As mentioned in \cite{yang2022anderson}, $C(a,b,m)$ is a monotonically decreasing function of $m$ for fixed $a$ and $b$, and decays exponentially to zero as $m$ goes to $\infty$.  We note that the estimation \eqref{eq:aAAsFPt} is more accurate for $s=1$, since \eqref{eq:aAAms} is not sharp. In the following, we first estimate $C(a,b,m)$ and then give a sufficient condition such that the subsequence $\{\mathbf{r}_{jp}\}$ of aAA($m$)[$1$]--FP[$t$] converges.

 \begin{lemma}\label{lem:upbd-C}
 	Let $a<b, \alpha=\frac{1}{b}$ and $\beta=\frac{1}{a}$. For $[\alpha,\beta]$ which does not include $0$ and 1, we have 
 	\begin{equation*}
 		\min_{p_{\ell}\in P_k, p_{\ell}(1)=1} \max_{t\in [\alpha,\beta]}|p_{\ell}(t)| =\frac{1}{|T_k(1+2\frac{1-\beta}{\beta-\alpha})|}\leq \epsilon(a,b,k),
 	\end{equation*}
 	where
 	\begin{equation*}
 		\epsilon(a,b,k)=
 		\begin{cases}
 			2\left(\frac{\sqrt{\frac{b(1-a)}{a(1-b)}}-1}{\sqrt{\frac{b(1-a)}{a(1-b)}}+1}\right)^k, & if \quad b<1\\
 			2\left(\frac{\sqrt{\frac{a(1-b)}{b(1-a)}}-1}{\sqrt{\frac{a(1-b)}{b(1-a)}}+1}\right)^k, & if \quad a>1.
 		\end{cases}
 	\end{equation*}
 \end{lemma}
 \begin{proof}
 	Note that $ -(1+2\frac{1-\beta}{\beta-\alpha})= 1+2\frac{\alpha-1}{\beta-\alpha}$. When $\alpha>1 (i.e., b<1)$, $1+2\frac{\alpha-1}{\beta-\alpha}>1$. Recall that for $|t|>1$
 	\begin{equation*}
 		T_k(t)=\frac{1}{2}\left( (t+\sqrt{t^2-1})^k+ (t+\sqrt{t^2-1})^{-k}\right)\geq\frac{1}{2} \left(t+\sqrt{t^2-1}\right)^k. 
 	\end{equation*}  
 	Let $\eta=\frac{\alpha-1}{\beta-\alpha}$. Thus,
 	\begin{equation*}
 		T_k(1+2\frac{\alpha-1}{\beta-\alpha})\geq \frac{1}{2}(1+2\eta + \sqrt{(1+2\eta)^2-1})^k=\frac{1}{2}\left(1+2\eta +2 \sqrt{\eta(1+\eta)}\right)^k. 	 
 	\end{equation*}
 	It can be shown that 
 	\begin{equation*}
 		1+2\eta +2 \sqrt{\eta(1+\eta)}= (\sqrt{\eta}+ \sqrt{1+\eta})^2=
 		\left(\sqrt{\frac{\alpha-1}{\beta-\alpha}}+ \sqrt{\frac{\beta-1}{\beta-\alpha}}\right)^2 =\frac{\sqrt{\frac{\beta-1}{\alpha-1}}+1}{\sqrt{\frac{\beta-1}{\alpha-1}}-1}.
 	\end{equation*}
 	Thus, 
 	\begin{equation*}
 		\frac{1}{T_k(1+2\frac{\alpha-1}{\beta-\alpha})}\leq 2\left(\frac{\sqrt{\frac{\beta-1}{\alpha-1}}-1}{\sqrt{\frac{\beta-1}{\alpha-1}}+1}\right)^k.
 	\end{equation*}
 	Substituting $\alpha =\frac{1}{b}$ and $\beta=\frac{1}{a}$ into the above inequality, we can obtain the desired result for $b<1$. 
 	When $\beta<1 (i.e., a>1)$, $\left|T_k(1+2\frac{1-\beta}{\beta-\alpha})\right|=T_k(1+2\frac{1-\beta}{\beta-\alpha})$. We take $\eta=\frac{1-\beta}{\beta-\alpha}$ in the above estimate, which will give the desired result.
 \end{proof}
 
 As discussed before, when $\|\mathbf{M}\|<1$, aAA($m$)[$s$]--FP[$t$] is convergent. Now, we explore the case where $\|\mathbf{M}\|>1$. Following Theorem \ref{thm:aAA-convergence} and Lemma \ref{lem:upbd-C}, we have the following sufficient convergence condition for aAA($m$)[$s$]--FP[$t$]. 
 \begin{corollary}\label{corol-noncontractive}
 	Consider aAA($m$)[$s$]--FP[$t$] with $m\leq t$. Assume that $\mathbf{M}$ is diagonalizable and all its eigenvalues are in the interval $[a, b]$, where $a>1$. If 
 	\begin{equation}\label{eq:suff-cond}
 		2\kappa_2(\mathbf{W})\left(\frac{\sqrt{\frac{a(1-b)}{b(1-a)}}-1}{\sqrt{\frac{a(1-b)}{b(1-a)}}+1}\right)^m \|\mathbf{M}\|^{t+s}<1,
 	\end{equation}
 	the subsequence $\{\mathbf{x}_{jp}\}$ generated by aAA($m$)[$s$]--FP[$t$] converges. 
 	Additionally, if $\mathbf{M}$ is symmetric and 
 		\begin{equation}\label{eq:suff-cond-sym}
 			\left(\frac{\sqrt{\frac{a(1-b)}{b(1-a)}}-1}{\sqrt{\frac{a(1-b)}{b(1-a)}}+1}\right)^m b^{t+s}<1,
 		\end{equation}
 		the subsequence $\{\mathbf{x}_{jp}\}$ generated by aAA($m$)[$s$]--FP[$t$] converges. 
 \end{corollary}
 	\begin{proof}
 		From the first inequality in \eqref{eq:rjp-rjm1p}, Lemma \ref{lem:upbd-C}, and the relation $p=t+s$, we derive the first result. If $\mathbf{M}$
 		is symmetric, then $\|\mathbf{M}\|=b$ and $\kappa_2(\mathbf{W})=1$, leading directly to the second result. 
 	\end{proof}
 Let us consider aAA($t$)[$1$]--FP[$t$] for symmetric $\mathbf{M}$.
 Then, \eqref{eq:suff-cond} is reduced to
 \begin{displaymath}\label{eq:suff-conds1m=t}
 	2\left(\frac{\sqrt{\frac{a(1-b)}{b(1-a)}}-1}{\sqrt{\frac{a(1-b)}{b(1-a)}}+1}\right)^t b^{t+1}<1,
 \end{displaymath}
 Table \ref{tab:Cheby-estimate} reports $C(a,b,m)b^{t+1}$ and $\epsilon(a,b,m)b^{t+1}$ with $t=m$ for different values of $a$ and $b$ to illustrate that even if the fixed-point iteration is divergent, aAA($m$)[$1$]--FP[$m$] may converge for large value of $m$. From Table \ref{tab:Cheby-estimate}, we see that $\epsilon(a,b,m)b^{t+1}$ is a sharp estimate of $C(a,b,m)b^{t+1}$. For $m=10$ and $15$, we see that aAA($m$)[$1$]--FP[$m$] converges. Increasing $m$ significantly improves the convergence rate when the underlying fixed-point iteration is convergent.
 
 \begin{table}[!ht]
 	\caption{The values of $C(a,b,m)b^{m+1}$ as well as $\epsilon(a,b,m)b^{m+1}$ (in bracket) for different values of $m$ and ($a,b$). The notation ``-" stands for values greater than one.}\label{tab:Cheby-estimate}
 	\begin{center}
 		\begin{tabular}{|c||c |c|c|c|c|}
 			\hline
 			$m$    & (0.3, 0.9)    &(1.5, 3)     &(2, 5)   &(10, 30)   &(20, 50)  \\
 			\hline
 			2     & 0.5134(0.6005)      &-   &-   &-   &-     \\
 			\hline
 			4   &0.1947(0.2003)   &0.4211(0.4211)   &-   &-      & -   \\
 			\hline
 			10     &0.0074(0.0074)    &0.0078(0.0078)   &0.0468(0.0468)    &0.1172(0.1172)      &0.0079(0.0079)     \\
 			\hline
 			15    & 0.0005(0.0005)       &0.0003(0.0003)    &0.0032(0.0032)    &0.0052(0.0052)    & 0.0001(0.0001)\\
 			\hline
 		\end{tabular}
 	\end{center}
 \end{table}
From the previous analysis of the linear case, we identify the crucial elements as the properties of the iteration matrix and the use of polynomial approximation techniques. These arguments, however, do not carry over directly to the nonlinear setting.
 
The main contributions  of Theorems \ref{thm:one-step-gain} and \ref{thm:aAA-convergence} establish general convergence results that do not require contractiveness, and Corollary \ref{corol-noncontractive} provides a new sufficient condition for convergence in the non‑contractive case.  We believe that these convergence results enrich the current understanding of AA‑type methods and offer guidance for their nonlinear extensions. When the iterates are sufficiently close to the exact solution, the nonlinear iteration can be locally characterized by its linearization, and analyzing this connection will be an important direction for our future work.
 
 \section{Numerical results}\label{sec:num}
 
 To illustrate the efficiency of the proposed aAA($m$)[$s$]--FP[$t$] method, we investigate a wide range of challenging problems, including linear cases with symmetric and nonsymmetric coefficient matrices, and nonlinear optimization problems. Although our analysis in the previous section focuses on $m\leq t$,  we remove this restriction for our numerical tests to show the flexibility of our algorithm. All tests below were run in MATLAB R2018b on an Ubuntu 18.04.1 LTS system with a 1.70 GHz Intel quad-core i5 processor and 8 GB RAM, with the CPU times reported in seconds. The computer implementation of the algorithms and numerical tests reported in this paper is freely available at \url{https://github.com/sleveque/aAA-FP}. Due to the complexity of the finite element implementation, the codes for the Poisson and the Navier--Stokes problems are available from the authors upon request.

 \subsection{Validation}
 In this subsection, we consider two examples to verify our theoretical result \eqref{eq:eqv-p-period}. In the first example, we consider the case of nonsymmetric $\mathbf{A}$; in the second, we study the case of symmetric $\mathbf{A}$. 
 
 In the following examples, rather than checking that the iterate \eqref{eq:eqv-p-period} holds for each $jp$ iteration, we consider the norm of the residual $\mathbf{r}_k$ obtained by aAA($\infty$)[1]--FP[$t$] and the norm of $\mathbf{M} \mathbf{r}_k^{G}$, with $\mathbf{r}_k^{G}$ being the residual obtained by GMRES, at each iteration $k$. In fact, from \eqref{eq:eqv-p-period} and $\mathbf{M}=\mathbf{I}-\mathbf{A}$ we can write $\mathbf{r}_{jp} = \mathbf{b} -\mathbf{A} \mathbf{x}_{jp} = \mathbf{b} - \mathbf{A} q(\mathbf{x}_{jp-1}^{G})=\mathbf{b}-\mathbf{A}(\mathbf{M}\mathbf{x}_{jp-1}^G+\mathbf{b})= \mathbf{M} (\mathbf{b} - \mathbf{A} \mathbf{x}_{jp-1}^{G})= \mathbf{M} \mathbf{r}_{jp-1}^{G}$. Thus, we expect a periodic alignment between the residual $\mathbf{r}_k$  and $\mathbf{M}\mathbf{r}_k^G$.

 \subsubsection{The nonsymmetric case}
 
 We first consider the case where $\mathbf{A}$ is nonsymmetric. We want to solve the system $\mathbf{A x} = \mathbf{b}$, where the matrix $\mathbf{A}\in \mathbb{R}^{n\times n}$ is the permutation  
 \begin{displaymath}
 	\mathbf{A} = \left[
 	\begin{array}{cccc}
 		0 & 0 & \ldots & 1  \\
 		1 & \ddots & & \vdots \\
 		0 & \ddots & & \vdots\\
 		0 & \ldots & 1 & 0
 	\end{array}
 	\right],
 \end{displaymath}
 and the vector $\mathbf{b}=[1,0,\ldots,0]^T$.  For this particular problem, GMRES stagnates with zero initial guess, but converges to the exact solution at the $n$th iteration. We compare GMRES with aAA($\infty$)[1]--FP[$t$] for $t=3$.  We use the vector of all ones as the initial guess $\mathbf{x}_0$. We opt for this choice so that GMRES does not stagnate.  In Figure \ref{fig:non-symmetric}, we plot the norm of $\mathbf{M}\mathbf{r}_k^{G}$ for $k=0,1, \cdots$  and the norm of the residual $\mathbf{r}_k$ starting from $k=0$ obtained by aAA($\infty$)[1]--FP[$3$] for $n=32$ (left panel) and $n=26$ (right panel). In the plots, we display $\|\mathbf{r}_k\|$ at position $k - 1$.

 \begin{figure}[H]
 	\centering \includegraphics[width=0.49\linewidth]{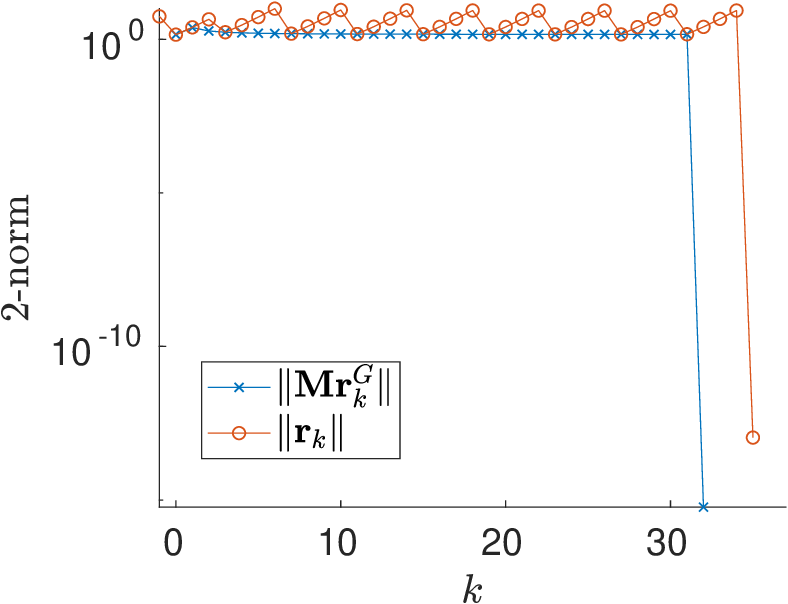}
 	\includegraphics[width=0.49\linewidth]{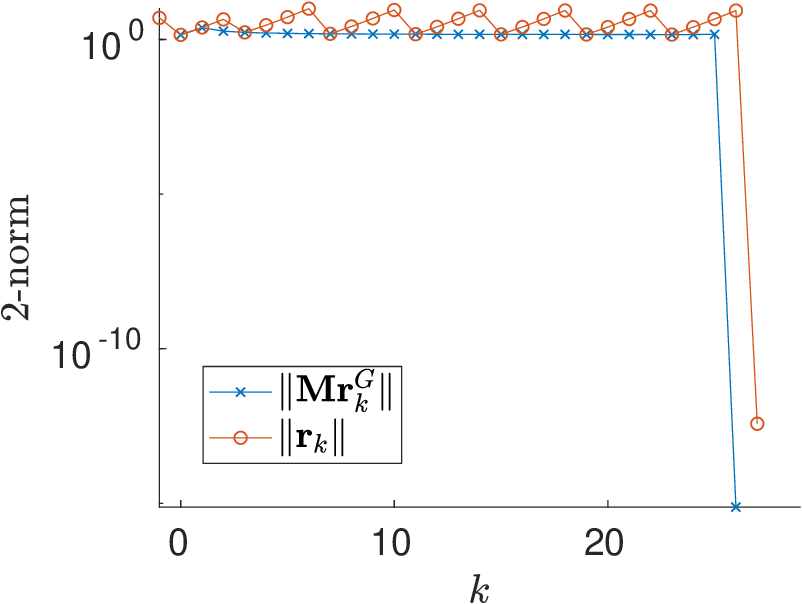}
 	\caption{Plots of the norms of $\mathbf{M}\mathbf{r}_k^{G}$ and $\mathbf{r}_k$ for the nonsymmetric case for  aAA($\infty$)[1]--FP[3]. Left: $n=32$. Right: $n=26$.}
 	\label{fig:non-symmetric}
 \end{figure}
 
 From the left panel of Figure \ref{fig:non-symmetric} (that is, $t=3$ and $p=t+1=4$), we observe  the predicted equivalence of the aAA($\infty$)[1]--FP[3] and the GMRES iterates, with period $p=4$ iteration, as described in \eqref{eq:eqv-p-period}. Note that the fixed-point iteration is divergent as $\|\mathbf{M}\|>1$; however, aAA($\infty$)[1]--FP[$3$] is still able to converge because one step AA gains more than three fixed-point iterations.  Since the result in \eqref{eq:eqv-p-period} only holds periodically with period $p$, we would like to understand what happens to aAA($\infty$)[1]--FP[$t$] when GMRES converges between two periods. For this reason, we set $n=26$ and compare GMRES with aAA($\infty$)[1]--FP[3] on the right panel of Figure \ref{fig:non-symmetric}. In this case, GMRES converges in $n=26$ iterations, but $n$ is not divisible by $p=4$.   As we can observe for $n=26$, aAA($\infty$)[1]--FP[3] converges in 28 iterations, even when GMRES converges between two periods. We would like to mention that, in Figure \ref{fig:non-symmetric}, the two iterative methods do not stagnate.

 \subsubsection{The symmetric case}
 We now consider the case where $\mathbf{A}$ is symmetric. We compare aAA($\infty$)[1]--FP[$t$] and GMRES in solving the following discretized  Poisson equation:
 \begin{displaymath}
 	\left\{
 	\begin{array}{cll}
 		- \Delta u &= &f(x_1,x_2)  \quad \mathrm{in} \; \Omega=[-1, 1]^2,\\
 		u &=& g(x_1,x_2)  \quad \mathrm{on} \; \partial \Omega.
 	\end{array}
 	\right.
 \end{displaymath}
 We set $f(x_1, x_2)=\frac{\pi^2}{2} \cos(\frac{\pi x_1}{2})\cos(\frac{\pi x_2}{2})$ and $g(x_1, x_2) = 1$, for which the analytical solution $u = \cos(\frac{\pi x_1}{2})\cos(\frac{\pi x_2}{2}) + 1$.
 
 We employ $Q_1$ finite elements as the discretization and construct a uniform mesh with 33 points in each direction. Upon discretization, we have to solve a system of the form $\mathbf{Au} = \mathbf{b}$, with $\mathbf{A}\in \mathbb{R}^{n \times n}$ symmetric positive definite. For this case, we have $n=31^2=961$ (the discretization does not take into account for the boundary nodes).  We run GMRES and aAA($\infty$)[1]--FP[$t$] for $t=0$ and $t=3$, using a reduction of $10^{-12}$ in the residual as the stopping criterion. Note that aAA($\infty$)[1]--FP[0] is AA($\infty$).  In Figure \ref{fig:poisson}, we plot the norm of $\mathbf{M}\mathbf{r}_k^{G}$, where $\mathbf{r}_k^{G}$ is the residual obtained by GMRES at the $k$th iteration, and the norm of the residual $\mathbf{r}_k$ obtained by aAA($\infty$)[1]--FP[$t$], with a zero initial guess for both methods. Again,   Figure \ref{fig:poisson} confirms \eqref{eq:eqv-p-period}.
 
 \begin{figure}[!htb]
 	\centering
 	\includegraphics[width=0.49\linewidth]{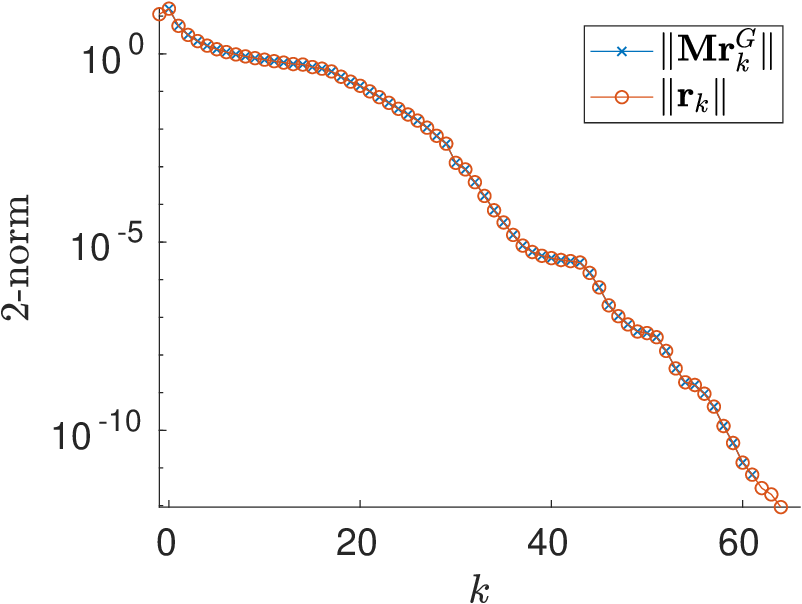}
 	\includegraphics[width=0.49\linewidth]{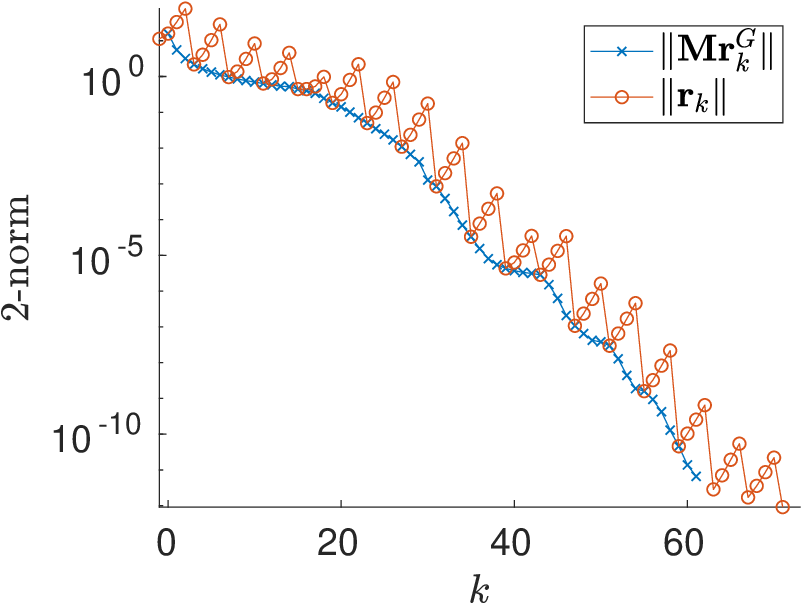}
 	\caption{Plots of the norms of $\mathbf{M} \mathbf{r}_k^{G}$ and $\mathbf{r}_k$ for the Poisson equation. Left: aAA($\infty$)[1]--FP[0]. Right: aAA($\infty$)[1]--FP[3].} 
 	\label{fig:poisson}
 \end{figure}

 \subsection{Application of aAA--FP to nonsymmetric cases}
 In this subsection, we apply aAA($m$)[$s$]--FP[$t$] to solve the nonsymmetric systems $\mathbf{A}\mathbf{x}=\mathbf{b}$ arising from finite elements discretizations. The problems we consider here are from the Matrix Market repository, which is available at \url{https://math.nist.gov/MatrixMarket/}. Specifically, we consider the problems: fidap008, fidap029, and fidapm37.  We apply AA($m$) and aAA($m$)[$s$]--FP[$t$] for different values of $m$, $s$, and $t$. We also consider standard GMRES as a solver. The fixed-point iteration for AA($m$) and aAA($m$)[$s$]--FP[$t$] is either weighted Jacobi with weight $\omega=0.5$ or Gauss--Seidel. For a fair comparison, we run GMRES for the solution of $\mathbf{P}^{-1}\mathbf{Ax} =\mathbf{P}^{-1}\mathbf{b}$, where $\mathbf{P}$ is either the diagonal of $\mathbf{A}$ for Jacobi, as done in \cite{pratapa2016anderson}, or its lower-triangular part. We allow a maximum of 5000 iterations and run the solvers until a reduction of $10^{-8}$ in the relative residual is achieved, that is, starting from an initial residual $\mathbf{r}_0$, we require that $\| \mathbf{r}_k \| / \| \mathbf{r}_0 \| < 10^{-8}$. For all the tests here, we use the vector of all ones as the initial guess. In Tables \ref{matrix_market1-100}, \ref{matrix_market2-60}, and \ref{matrix_market2-GS}, we report the number of iterations and the CPU times required for solving the problems considered here; further in Table \ref{matrix_market1-100} we report the dimension of each system to be solved together with the condition number of $\mathbf{A}$ and $\mathbf{P}^{-1}\mathbf{A}$ for Jacobi. We note that the Jacobi iteration, viewed as a fixed‑point method, is divergent in all cases, and that the Gauss–Seidel method, also viewed as a fixed‑point iteration, fails to meet the stopping criterion within 5000 iterations in every case.
 	
 	\begin{table}[!ht]
 		\caption{Degrees of freedom $n$, the condition number of $\mathbf{A}$ and of $\mathbf{P}^{-1}\mathbf{A}$, denoted by $\kappa$, number of iterations $\texttt{it}$, and CPU time for AA(100),  aAA(100)[10]--FP[5], and GMRES(100), when applying Jacobi using all ones as the initial guess. The notation $\dagger$ indicates that the solver did not achieve the desired stopping criterion within the 5000 iterations.}\label{matrix_market1-100}
 		\begin{center}
 			\resizebox{1\textwidth}{!}{		
 				\begin{tabular}{|c||c|c|c||c|c||c|c||c|c|}
 					\hline
 					& & & & \multicolumn{2}{c||}{AA(100)}   & \multicolumn{2}{c||}{aAA(100)[10]--FP[5]} & \multicolumn{2}{c|}{GMRES(100) } \\
 					\cline{5-10}
 					problem & $n$ & $\kappa(\mathbf{A})$ & $\kappa(\mathbf{P}^{-1}\mathbf{A})$  & $\texttt{it}$ & CPU & $\texttt{it}$ & CPU & $\texttt{it}$ & CPU \\
 					\hline
 					fidap008 & 3096 & 4.48e+10 & 7.27e+08 &  5000$\dagger$ &  46.8 &  5000$\dagger$ & 37.8 & 5000$\dagger$ & 2.90 \\
 					\hline
 					fidap029 & 2870 & 7.52e+02 & 5.83e+00 &  23 & 0.02   & 23  & 0.014 & 22  & 0.06  \\
 					\hline
 					fidapm37 & 9152 & 8.59e+09 & 2.84e+07   & 5000$\dagger$  & 226   &  5000$\dagger$ & 182  & 5000$\dagger$  & 20.4  \\
 					\hline
 			\end{tabular}}
 		\end{center}
 	\end{table}

 	\begin{table}[!ht]
 		\caption{Number of iterations $\texttt{it}$ and CPU time for  AA($\infty$),  AA(60), aAA(60)[8]--FP[4], aAA(60)[10]--FP[5], and aAA($\infty$)[10]--FP[5], when applying Jacobi using all ones as the initial guess. The notation $\dagger$ indicates that the solver did not achieve the desired stopping criterion within the 5000 iterations. The notation $\ddagger$ means that Matlab reached the maximum allowed CPU time for the run.}\label{matrix_market2-60}
 		\begin{center}
 			\resizebox{1\textwidth}{!}{ 
 				\begin{tabular}{|c||c|c||c|c||c|c||c|c||c|c|}
 					\hline
 					& \multicolumn{2}{c||}{AA($\infty$)} & \multicolumn{2}{c||}{AA(60)}  & \multicolumn{2}{c||}{aAA(60)[8]--FP[4]} & \multicolumn{2}{c||}{aAA(60)[10]--FP[5]} & \multicolumn{2}{c|}{aAA($\infty$)[10]--FP[5]} \\
 					\cline{2-11}
 					problem & $\texttt{it}$ & CPU & $\texttt{it}$ & CPU & $\texttt{it}$ & CPU & $\texttt{it}$ & CPU & $\texttt{it}$ & CPU\\
 					\hline
 					fidap008 & 3099   & 7857    &  5000$\dagger$ & 25.5   &  5000$\dagger$ & 20.9   & 5000$\dagger$  & 22.1 & 5000$\dagger$  & 22,196   \\
 					\hline
 					fidap029 &  23 &  0.022  &  23 & 0.02   &  23 & 0.015   &  23 & 0.015 &  23 & 0.016   \\
 					\hline
 					fidapm37 & 2396  & 14,153    &  5000$\dagger$ & 100   & 5000$\dagger$  & 87.5   & 5000$\dagger$  & 89.4 & $\ddagger$  & $\ddagger$   \\
 					\hline
 			\end{tabular}}
 		\end{center}
 	\end{table}

 	\begin{table}[!ht]
 		\caption{Number of iterations $\texttt{it}$ and CPU time for  AA($\infty$),  AA(60), aAA(60)[10]--FP[5], aAA($\infty$)[10]--FP[5], GMRES(60), and  GMRES($\infty$), when applying Gauss--Seidel using all ones as the initial guess. The notation $\dagger$ indicates that the solver did not achieve the desired stopping criterion within the 5000 iterations.}\label{matrix_market2-GS}
 		\begin{center}
 			\resizebox{1\textwidth}{!}{ 
 				\begin{tabular}{|c||c|c||c|c||c|c||c|c||c|c||c|c|}
 					\hline
 					& \multicolumn{2}{c||}{AA($\infty$)} & \multicolumn{2}{c||}{AA(60)}  &  \multicolumn{2}{c||}{aAA(60)[10]--FP[5]} & \multicolumn{2}{c||}{aAA($\infty$)[10]--FP[5]} & \multicolumn{2}{c||}{GMRES(60)} & \multicolumn{2}{c|}{GMRES($\infty$)} \\
 					\cline{2-13}
 					problem & $\texttt{it}$ & CPU & $\texttt{it}$ & CPU & $\texttt{it}$ & CPU & $\texttt{it}$ & CPU & $\texttt{it}$ & CPU & $\texttt{it}$ & CPU\\
 					\hline
 					fidap008 & 3371  & 9892   &   5000$\dagger$  & 23.9  & 5000$\dagger$  & 18.2 &  3324 &  6513  & 5000$\dagger$  & 27.4 & 2315  & 36.7   \\
 					\hline
 					fidap029 &  18  & 0.012   & 18  & 0.012   &  18  & 0.019 &  18 & 0.011  & 19  & 0.38  & 18  & 0.012   \\
 					\hline
 					fidapm37 & 1603  & 4437  & 5000$\dagger$  & 84.19   & 5000$\dagger$  & 84.1 & 1778  & 4164 & 5000$\dagger$  & 142 &  1826 &  141  \\
 					\hline
 			\end{tabular}}
 		\end{center}
 	\end{table}

 	From Tables \ref{matrix_market1-100}, \ref{matrix_market2-60} and \ref{matrix_market2-GS}, we can observe that aAA($m$)[$s$]--FP[$t$] is able to outperform AA($m$) and GMRES($m$) for problem fidap029. In fact, restarted GMRES($m$) with restart $m=100$ takes 22 iterations with 0.06 seconds for this problem. On the other hand, aAA(100)[10]--FP[5] requires 0.014 seconds for reaching convergence, and is about two times faster than AA(100). For problems fidap008 and fidapm37, the solvers are not able to converge to the prescribed accuracy within 5000 iterations (reach $10^{-5}$), and one has to employ full GMRES or window size $m=\infty$ within Anderson approach. This is true also when applying Gauss--Seidel as a fixed-point iteration. We believe that this is the case because the preconditioners employed are not well suited for these two problems. We would like to mention that, for Jacobi, our results differ from the ones presented in \cite{pratapa2016anderson}; this is because the method presented in \cite{pratapa2016anderson} includes a relaxation parameter in the Anderson acceleration, which we do not explore.

 	For these tests, the results suggest that effective preconditioner design, when combined with AA-type methods, is necessary to achieve fast convergence, which we leave for future work. Conversely, AA-type methods may be better suited for nonlinear problems, as demonstrated by our simulations in later sections.

 \subsection{Application of aAA--FP to the incompressible Navier--Stokes equations}
 In this section, we present the effect of AA and aAA--FP on accelerating the Picard iteration for the solution of the incompressible Navier--Stokes equations.
 
 Given a spatial domain $\Omega \subset \mathbb{R}^d$, with $d \in \{2,3\}$, the incompressible Navier--Stokes equations are given by
 \begin{displaymath}
 	\left\{
 	\begin{array}{rr}
 		- \nu \nabla^2 \vec{u} + \vec{u} \cdot \nabla \vec{u} + \nabla p = \vec{f}, & \mathrm{in} \; \Omega, \\
 		- \nabla \cdot \vec{u} = 0, & \mathrm{in} \; \Omega, \\
 		\vec{u} = \vec{g}, & \mathrm{on} \; \partial \Omega, 
 	\end{array}
 	\right.
 \end{displaymath}
 where the force function $\vec{f}$ and the boundary conditions $\vec{g}$ are given. The value $\nu$ is called the viscosity of the fluid. The flow described by the Navier--Stokes equations is influenced by the viscosity. In fact, the type of flow is determined by the Reynolds number $Re=\frac{L V}{\nu}$, where $L$ and $V$ are the characteristic length and velocity scales of the fluid, respectively. Specifically, for small Reynolds number, the flow is laminar while becoming turbulent at high Reynolds number. We consider $\Omega=(0,1)^2$, $\vec{f}=0$, and
 \begin{displaymath}
 	\vec{g} = \left\{
 	\begin{array}{ll}
 		\left[1,0\right]^T & \mathrm{on} \: \partial \Omega_1 = \left(0,1 \right)\times \left\{1\right\},\\
 		\left[0,0\right]^T & \mathrm{on} \:\partial \Omega \setminus \partial\Omega_1.
 	\end{array}
 	\right.
 \end{displaymath}
 
 Starting from an approximation $\vec{u}^{\, (k)}$ and $p^{(k)}$ of $\vec{u}$ and $p$, respectively, at each Picard step, one solves the following Oseen problem:
 \begin{displaymath}
 	\left\{
 	\begin{array}{rr}
 		- \nu \nabla^2 \vec{\delta u}^{(k)} + \vec{u}^{\, (k)} \cdot \nabla \vec{\delta u}^{(k)} + \nabla \delta p^{(k)} = \hat{f}, & \mathrm{in} \; \Omega, \\
 		- \nabla \cdot \vec{\delta u}^{(k)} = 0, & \mathrm{in} \; \Omega, \\
 		\vec{\delta u}^{(k)} = 0, & \mathrm{on} \; \partial \Omega.
 	\end{array}
 	\right.
 \end{displaymath}
 Then, the new approximation is given by $\vec{u}^{\, (k+1)} = \vec{u}^{\, (k)} + \vec{\delta u}^{(k)}$ and $p^{(k + 1)} = p^{(k)} + \delta p^{(k)}$. In the Oseen problem, the function $\hat{f}$ contains information on the nonlinear residual.
 
 Employing inf--sup stable finite elements as discretization, at each Picard step one has to solve a saddle-point system of the form
 \begin{equation}\label{Picard_linearization}
 	\left[
 	\begin{array}{cc}
 		\mathbf{A} & \mathbf{B}^T\\
 		\mathbf{B} & \mathbf{0}
 	\end{array}
 	\right]
 	\left[
 	\begin{array}{c}
 		\boldsymbol{\delta u}^{(k)} \\
 		\boldsymbol{\delta p}^{(k)}
 	\end{array}
 	\right] =
 	\left[
 	\begin{array}{c}
 		\boldsymbol{\hat{f}} \\
 		\boldsymbol{0}
 	\end{array}
 	\right].
 \end{equation}
 Here, we set $\mathbf{A}=\nu \mathbf{K} + \mathbf{N}$, with $\mathbf{K}$ the vector-stiffness matrix and $\mathbf{N}$ the vector-convection matrix, while $\mathbf{B}$ is the (negative) divergence matrix. As we mentioned above, at high Reynolds number the flow becomes turbulent. In this case, the problem is convection-dominated, and one requires to employ a stabilization technique. For this test, we employ the Local Projection Stabilization (LPS) technique described in \cite{Becker_Braack,Becker_Vexler,Braack_Burman}. In this test, we employ inf--sup stable Taylor--Hood $[{Q}_2]^2$--${Q}_1$ finite elements as discretization in the spatial dimensions.
 
 We now test how AA($m$) and aAA($m$)[$s$]--FP[$t$] can accelerate the convergence of the Picard iteration. We employ a level of refinement $l=6$, represented by a uniform grid of size $129 \times 129$ for the velocity and of size $65 \times 65$ for the pressure, while letting varying the viscosity $\nu$. We run AA($m$), aAA($m$)[$s$]--FP[$t$], and the Picard iterations until a reduction of $10^{-10}$ in the residual of the fixed point iteration is obtained. For the results presented here, each Picard linearization of the Navier--Stokes equations (that is, the linear system in \eqref{Picard_linearization}) is solved exactly, by employing Matlab's backslash. We employ the zero vector as the initial guess for the tests in this subsection. We report in Figure \ref{fig:Navier-Stokes1} the results with $\nu=\frac{1}{500}$, and in Figure \ref{fig:Navier-Stokes2} the results with $\nu=\frac{1}{2500}$, for $m=1$, 5, and 10.  From Figures \ref{fig:Navier-Stokes1}--\ref{fig:Navier-Stokes2} we can observe that the number of iterations required by aAA($m$)[$s$]--FP[$t$] is comparable with the number of iterations required by AA($m$) for reaching the prescribed reduction of the nonlinear residual.

 \begin{figure}[!htb]
 	\begin{subfigure}[b]{0.49\textwidth}
 		\centering
 		\includegraphics[width=1.\linewidth]{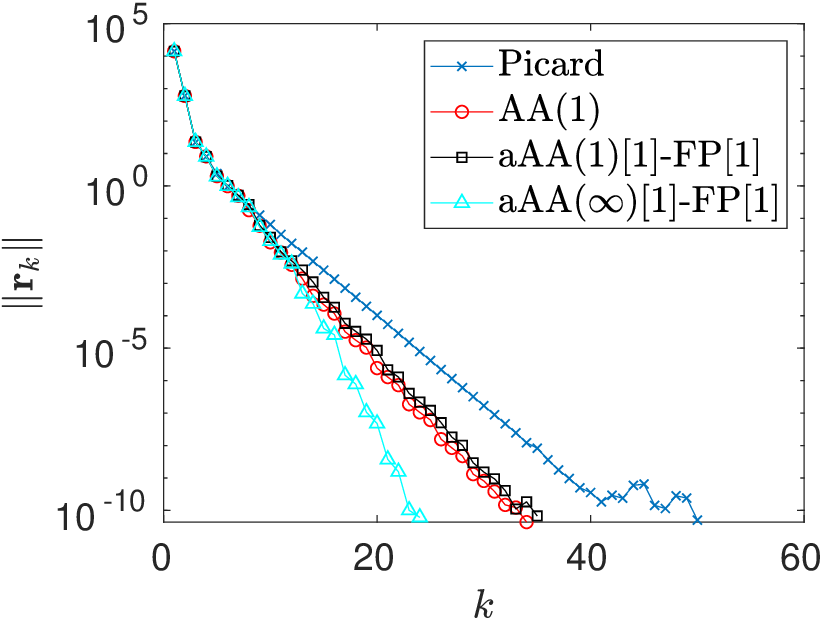}
 		\caption{$m=t=1$ and aAA($\infty$)[1]--FP[1].}
 	\end{subfigure}
 	\begin{subfigure}[b]{0.49\textwidth}
 		\centering
 		\includegraphics[width=1.\linewidth]{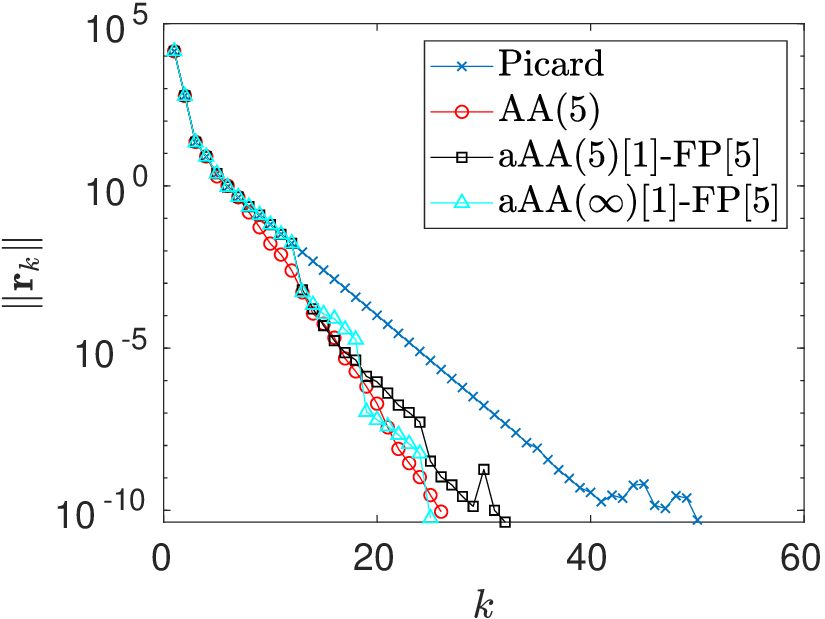}
 		\caption{$m=t=5$ and aAA($\infty$)[1]--FP[5].}
 	\end{subfigure}
 	\begin{subfigure}[b]{0.49\textwidth}
 		\centering
 		\includegraphics[width=1.\linewidth]{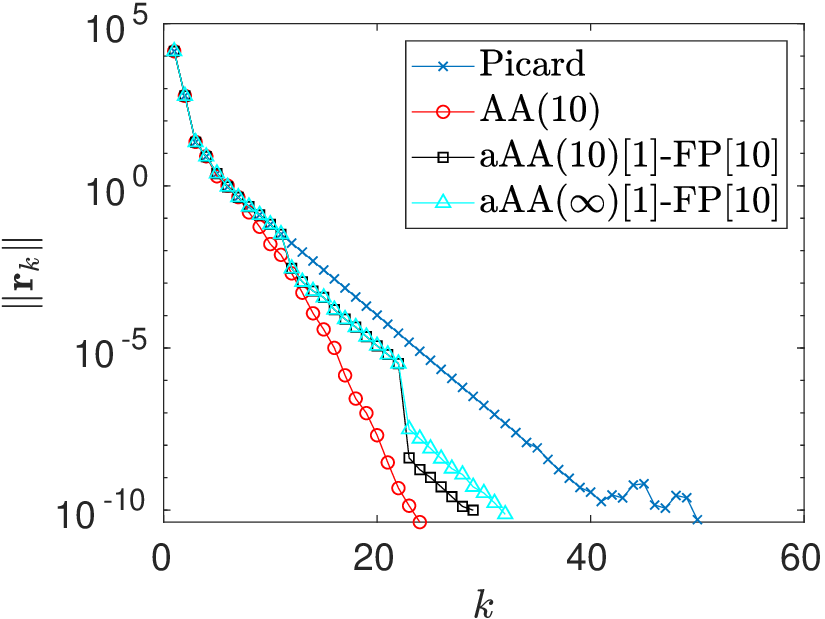}
 		\caption{$m=t=10$ and aAA($\infty$)[1]--FP[10].}
 	\end{subfigure}
 	\begin{subfigure}[b]{0.49\textwidth}
 		\centering
 		\includegraphics[width=1.\linewidth]{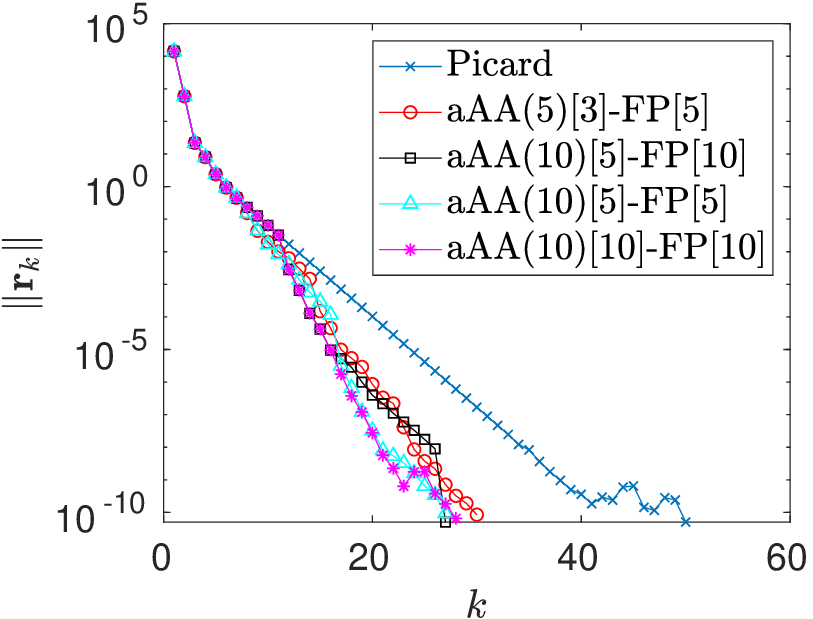}
 		\caption{Mix values.}
 	\end{subfigure}
 	\centering
 	\caption{Comparison of AA($m$), aAA($m$)[1]--FP[$t$], and Picard iteration, with level of refinement $l=6$ and $\nu=\frac{1}{500}$, for $m=1$, 5, 10, and $\infty$, and $t=1$, 5, and 10.}\label{fig:Navier-Stokes1}
 \end{figure}
 
 \begin{figure}[!htb]
 	\begin{subfigure}[b]{0.49\textwidth}
 		\centering
 		\includegraphics[width=1.\linewidth]{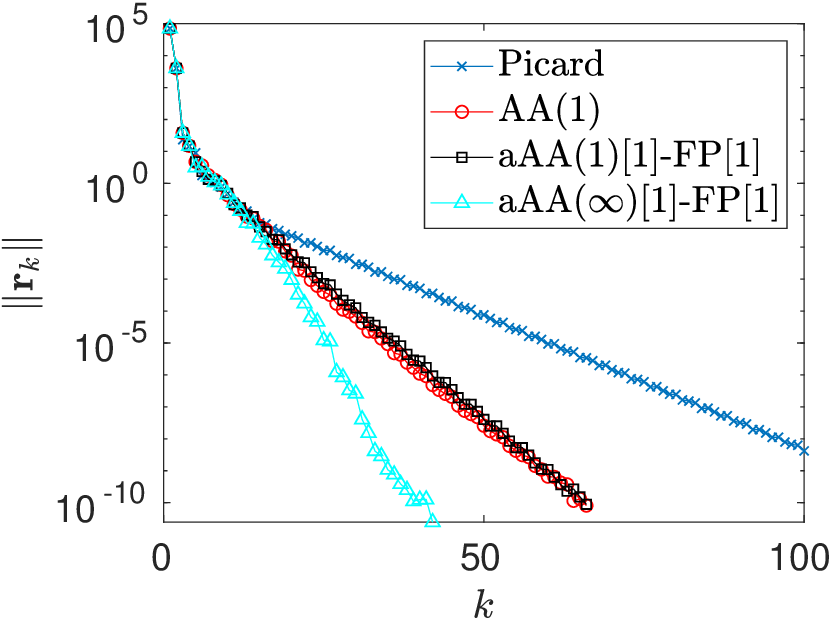}
 		\caption{$m=1$.}
 	\end{subfigure}
 	\begin{subfigure}[b]{0.49\textwidth}
 		\centering
 		\includegraphics[width=1.\linewidth]{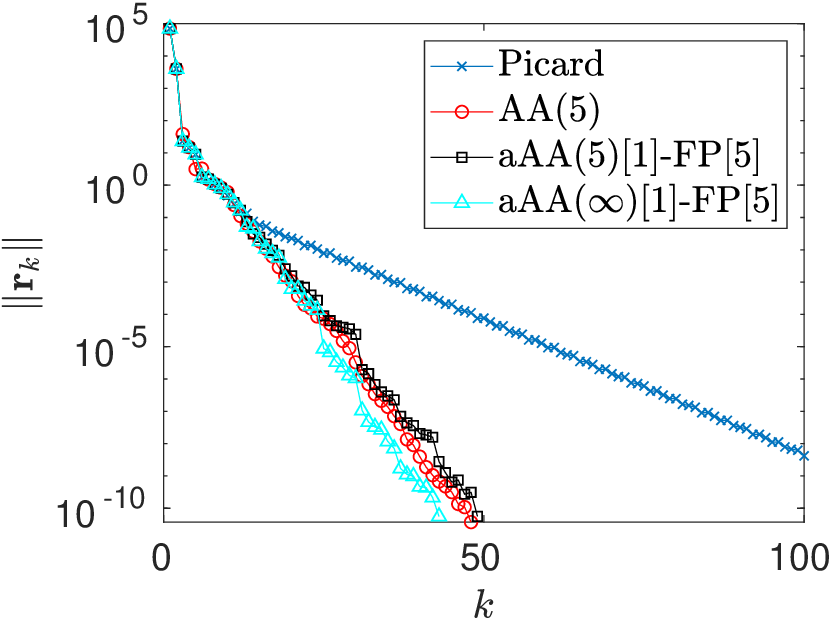}
 		\caption{$m=5$.}
 	\end{subfigure}
 	\begin{subfigure}[b]{0.49\textwidth}
 		\centering
 		\includegraphics[width=1.\linewidth]{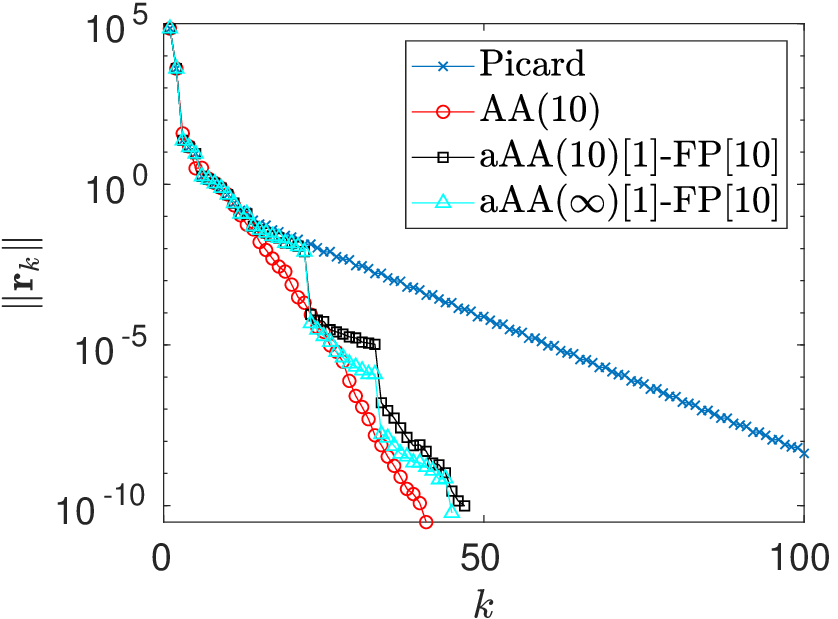}
 		\caption{$m=10$.}
 	\end{subfigure}
 	\begin{subfigure}[b]{0.49\textwidth}
 		\centering
 		\includegraphics[width=1.\linewidth]{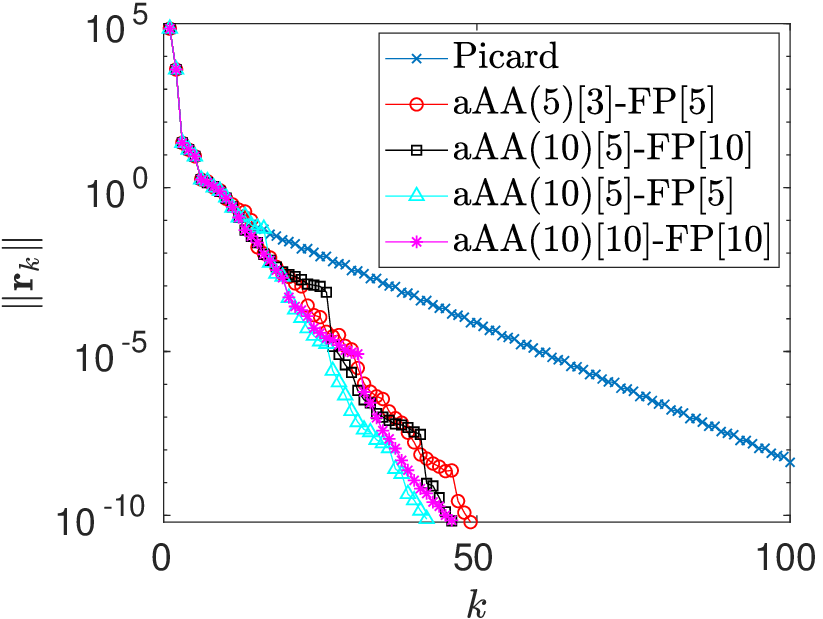}
 		\caption{Mix values.}
 	\end{subfigure}
 	\centering
 	\caption{Comparison of AA($m$), aAA($m$)[1]--FP[$m$], and Picard iteration, with level of refinement $l=6$ and $\nu=\frac{1}{2500}$, for $m=1$, 5, and 10.}\label{fig:Navier-Stokes2}
 \end{figure}

 Further, we report in Table \ref{table:AAPicard} the number of iterations together with the CPU time required for reaching the prescribed accuracy in the case of $\nu=\frac{1}{2500}$ for AA($m$) and aAA($m$)[$s$]--FP[$t$] methods with different values of $m$, $s$, and $t$. We mention that the CPU times include the time required for the construction of the finite element matrices and of the nonlinear residual, aside the times for solving the linear system in \eqref{Picard_linearization}. We mention here that the construction of the stabilization matrix at each non-linear iteration is the most expensive work, taking about 60\% of the total CPU times.
 
 First, we run AA($\infty$) and it takes 39 iterations with 11,281 seconds. We mention that the Picard iteration does not converge in 100 iterations, reaching a tolerance of $4.2 \cdot 10^{-9}$ in 27,999 seconds. These results will be a reference for aAA-FP and windowed AA. We show that aAA-FP is competitive with AA in Table \ref{table:AAPicard}: aAA(20)[20]--FP[20] and aAA(20)[25]--FP[20] outperform aAA($20$) in terms of CPU time, aAA($20$) outperforms AA($\infty$) in terms of CPU time and iteration number. For other values of  $m$, $s$, and $t$ considered here, aAA($m$)[$s$]--FP[$t$] requires between 4 and 13 iterations more than AA($m$) to reach convergence. All the tests for aAA-FP and AA($m$) accelerate Picard iteration in terms of both CPU time and iteration number.  aAA($m$)[$s$]--FP[$t$] outperforms aAA($m$)[1]--FP[$t$] in terms of CPU time. Finally, we would like to note that, as opposed to common opinion, for this particular problem finite windowed AA($m$) performs better than AA($\infty$).

 \begin{table}[!ht]
 	\caption{Number of iterations and CPU time for AA($m$) and aAA($m$)[$s$]--FP[$t$] methods to reach the prescribed accuracy when accelerating the Picard iteration with $\nu = \frac{1}{2500}$ for different values of $m$, $s$, and $t$.}\label{table:AAPicard}
 	\begin{center}
 		{\begin{tabular}{|cc|cc|cc|cc|}
 				\hline
 				\multicolumn{2}{|c|}{AA(10)} & \multicolumn{2}{c|}{AA(15)} & \multicolumn{2}{c|}{AA(20)}   \\
 				\hline
 				$\texttt{it}$ & CPU & $\texttt{it}$ & CPU & $\texttt{it}$ & CPU \\
 				41 & 11,837 & 38 & 10,982 &  \cellcolor{pink} 37 &  \cellcolor{green}10,495   \\
 				\hline
 				\multicolumn{2}{|c|}{aAA(10)[1]--FP[5]} & \multicolumn{2}{c|}{aAA(15)[1]--FP[8]} & \multicolumn{2}{c|}{aAA(20)[1]--FP[10]} \\
 				\cline{1-6}
 				$\texttt{it}$ & CPU & $\texttt{it}$ & CPU & $\texttt{it}$ & CPU \\
 				46  & 14,413  & 46  & 14,322  & 45 & 14,058 \\
 				\cline{1-6}
 				\multicolumn{2}{|c|}{aAA(10)[5]--FP[10]} & \multicolumn{2}{c|}{aAA(15)[10]--FP[15]} & \multicolumn{2}{c|}{aAA(20)[10]--FP[20]} \\
 				\cline{1-6}
 				$\texttt{it}$ & CPU & $\texttt{it}$ & CPU & $\texttt{it}$ & CPU \\
 				46 & 13,084 & 43 & 12,513 & 50 & 14,272 \\
 				\cline{1-6}
 				\multicolumn{2}{|c|}{aAA(10)[10]--FP[10]} & \multicolumn{2}{c|}{aAA(15)[15]--FP[15]} & \multicolumn{2}{c|}{aAA(20)[20]--FP[20]} \\
 				\cline{1-6}
 				$\texttt{it}$ & CPU & $\texttt{it}$ & CPU & $\texttt{it}$ & CPU \\
 				46  & 12,965 & 47  & 13,374  &  37 & 10,452 \\
 				\cline{1-6}
 				\multicolumn{2}{|c|}{aAA(10)[15]--FP[10]} & \multicolumn{2}{c|}{aAA(15)[20]--FP[15]} & \multicolumn{2}{c|}{aAA(20)[25]--FP[20]} \\
 				\cline{1-6}
 				$\texttt{it}$ & CPU & $\texttt{it}$ & CPU & $\texttt{it}$ & CPU \\
 				45 & 12,724 & 44 & 12,732 &  \cellcolor{pink} 37  &  \cellcolor{green} 10,278  \\
 				\cline{1-6}
 		\end{tabular}}
 	\end{center}
 \end{table}

 \subsection{Application of aAA--FP to nonlinear optimization}
 
 In this subsection, we consider problems from constrained optimization. In particular, we show how aAA($m$)[$s$]--FP[$t$] can significantly speed up ADMM \cite{ADMM}.
 
 Given the separable constrained optimization problem
 \begin{displaymath}
 	\min_{\mathbf{x}, \mathbf{y}} f_1(\mathbf{x}) + f_2(\mathbf{y}),
 \end{displaymath}
 such that
 \begin{displaymath}
 	\mathbf{A} \mathbf{x} + \mathbf{B} \mathbf{y} = \mathbf{b},
 \end{displaymath}
 ADMM solves a sequence of subproblems in the primal variables and then updates the Lagrange multiplier. Specifically, after introducing the Lagrange multiplier $\bm{\lambda}$, one considers the Lagrangian
 \begin{displaymath}
 	L_{\mu}(\mathbf{x}, \mathbf{y}, \bm{\lambda}) = f_1(\mathbf{x}) + f_2(\mathbf{y}) + \bm{\lambda}^T (\mathbf{A} \mathbf{x} + \mathbf{B} \mathbf{y}- \mathbf{b}) + \dfrac{\mu}{2} \|\mathbf{A} \mathbf{x} + \mathbf{B} \mathbf{y} - \mathbf{b} \|^2_2,
 \end{displaymath}
 where $\mu>0$ is a penalty parameter. Then, starting from approximations $\mathbf{y}^{(k)}$ and $\bm{\lambda}^{(k)}$, ADMM alternates the minimization of the Lagrangian with respect to the primal variable $\mathbf{x}$ and $\mathbf{y}$, then updates $\bm{\lambda}$, as follows:
 \begin{displaymath}
 	\left\{
 	\begin{array}{l}
 		\vspace{1ex}
 		\mathbf{x}^{(k+1)} = \argmin_{\mathbf{x}} L_{\mu}(\mathbf{x}, \mathbf{y}^{(k)}, \bm{\lambda}^{(k)}) \\
 		\vspace{1ex}
 		\mathbf{y}^{(k+1)} = \argmin_{\mathbf{y}} L_{\mu}(\mathbf{x}^{(k+1)}, \mathbf{y}, \bm{\lambda}^{(k)})\\
 		\bm{\lambda}^{(k+1)} = \bm{\lambda}^{(k)} + \mu (\mathbf{A} \mathbf{x}^{(k+1)} + \mathbf{B} \mathbf{y}^{(k+1)} - \mathbf{b}).
 	\end{array}
 	\right.
 \end{displaymath}
 By introducing the variable $\bar{\bm{\lambda}}:= \frac{1}{\mu} \bm{\lambda}$, one can rewrite the Lagrangian $L_{\mu}(\mathbf{x}, \mathbf{y}, \bm{\lambda})$ in the following equivalent form:
 \begin{displaymath}
 	L_{\mu}(\mathbf{x}, \mathbf{y}, \bar{\bm{\lambda}}) = f_1(\mathbf{x}) + f_2(\mathbf{y}) + \dfrac{\mu}{2} \| \mathbf{A} \mathbf{x} + \mathbf{B} \mathbf{y} - \mathbf{b} + \bar{\bm{\lambda}} \|^2_2 - \dfrac{\mu}{2} \| \bar{\bm{\lambda}} \|^2_2.
 \end{displaymath}
 Then, starting from $\mathbf{y}^{(k)}$ and $\bar{\bm{\lambda}}^{(k)}$, one can obtain that the ADMM steps are given by
 \begin{equation}\label{ADMM_scaled}
 	\left\{
 	\begin{array}{l}
 		\vspace{1ex}
 		\mathbf{x}^{(k+1)} = \argmin_{\mathbf{x}} f_1(\mathbf{x}) + \dfrac{\mu}{2} \| \mathbf{A} \mathbf{x} + \mathbf{B} \mathbf{y}^{(k)} - \mathbf{b} + \bar{\bm{\lambda}}^{(k)} \|^2_2 \\
 		\vspace{1ex}
 		\mathbf{y}^{(k+1)} = \argmin_{\mathbf{y}} f_2(\mathbf{y}) + \dfrac{\mu}{2} \| \mathbf{A} \mathbf{x}^{(k+1)} + \mathbf{B} \mathbf{y} - \mathbf{b} + \bar{\bm{\lambda}}^{(k)} \|^2_2\\
 		\bar{\bm{\lambda}}^{(k+1)} = \bar{\bm{\lambda}}^{(k)} + \mathbf{A} \mathbf{x}^{(k+1)} + \mathbf{B} \mathbf{y}^{(k+1)} - \mathbf{b}.
 	\end{array}
 	\right.
 \end{equation}
 Convergence is monitored by means of primal feasibility
 \begin{equation}\label{eq:primal-feasibility}
 	\mathbf{r}^{(k+1)}_p = \mathbf{A} \mathbf{x}^{(k+1)} + \mathbf{B} \mathbf{y}^{(k+1)} - \mathbf{b}
 \end{equation}
 and dual feasibility
 \begin{equation}\label{eq:dual-feasibility}
 	\mathbf{r}^{(k+1)}_d = \mu \mathbf{A}^T\mathbf{B} (\mathbf{y}^{(k+1)} - \mathbf{y}^{(k)}),
 \end{equation}
 as those two residuals converge to zero as ADMM proceeds, see, for example, \cite{ADMM}.
 
 ADMM can be reformulated as a fixed-point iteration of the form in \eqref{eq:FP} by setting $\mathbf{x}_k=[\mathbf{y}^{(k)}, \bar{\bm{\lambda}}^{(k)}]$, solving the first step in \eqref{ADMM_scaled}, and then considering $q(\cdot)$ as the result obtained by the second and third steps in \eqref{ADMM_scaled}.
 
 In what follows, we show that accelerating ADMM with aAA($m$)[$s$]--FP[$t$] is more beneficial than applying the simple AA($m$) to accelerate ADMM. We mention here that a proper study on the speed-up obtained when applying AA($m$) to accelerate ADMM was done in \cite{wang2021asymptotic}. As opposed to the work \cite{wang2021asymptotic}, where the convergence of the method is monitored by checking the norm of the primal feasibility in \eqref{eq:primal-feasibility} and the dual feasibility in \eqref{eq:dual-feasibility}, in our tests we monitor the norm of the residual $r(\mathbf{x}_k)$ defined in \eqref{eq:kthresidual}. We would like to mention that we do not observe a substantial difference in the solutions obtained by the two stopping criterion. We run the ADMM, AA($m$), and aAA($m$)[$s$]--FP[$t$] up to a tolerance of $10^{-12}$ is achieved.
 
 We consider three different problems. The first two are the total variation and lasso problems, which are classical nonlinear and nonsmooth problems arising in optimization. The last problem is a nonnegative least squares problem, a nonlinear problem with inequality constraints.

 \subsubsection{Total variation}\label{sec:TV}
 We consider here the following total variation problem:
 \begin{equation}\label{eq:TV}
 	\min_{\mathbf{x}} \dfrac{1}{2} \| \widehat{\mathbf{x}} - \mathbf{x}\|^2_2 + \beta \| \mathbf{G} \mathbf{x} \|^2_1,
 \end{equation}
 where $\widehat{\mathbf{x}}\in \mathbb{R}^n$ is a given vector, $\beta>0$ is a smoothing parameter, and $\mathbf{G} \in \mathbb{R}^{(n-1)\times n}$ is given by
 \begin{displaymath}
 	\mathbf{G}=\left[
 	\begin{array}{ccccc}
 		-1 & 1 & \\
 		& \ddots & \ddots \\
 		& & -1 & 1
 	\end{array}
 	\right].
 \end{displaymath}
 Note that \eqref{eq:TV} is nonlinear and nonsmooth, and it can be reformulated within the ADMM framework by introducing the variable $\mathbf{y} = \mathbf{G} \mathbf{x}$, resulting in the following minimization problem:
 \begin{displaymath}
 	\min_{\mathbf{x}, \mathbf{y}} \dfrac{1}{2} \| \widehat{\mathbf{x}} - \mathbf{x}\|^2_2 + \beta \| \mathbf{y} \|^2_1,
 \end{displaymath}
 such that
 \begin{displaymath}
 	\mathbf{G} \mathbf{x} - \mathbf{y} = \mathbf{0}.
 \end{displaymath}
 
 Following the work in \cite{wang2021asymptotic}, we set $n=1000$ and take the vector $\widehat{\mathbf{x}}$ to be randomly generated from the standard normal distribution. Further, we set the smoothing parameter $\beta=0.001 \| \widehat{\mathbf{x}} \|_{\infty}$ and the penalty parameter $\mu=10$. For this test, we employ the zero vector as the initial guess.
 
 We apply aAA($m$)[$s$]--FP[$t$] to accelerate the fixed-point iteration generated by ADMM, for different values of $m$, $s$, and $t$, and compare the results with pure ADMM and AA($m$). We allow for 1000 iterations and set as stopping criterion a reduction of $10^{-12}$ in the residual. In Table \ref{table:TV1}, we report the number of iterations and the CPU times required by each solver for different values of $m$, $s$, and $t$. Further, we report in Figure \ref{fig:TV} the history of the residual norm, for ADMM, A($m$), and aAA($m$)[$s$]--FP[$t$], for various values of $m$, $s$, and $t$.

 \begin{table}[!ht]
 	\caption{Total variation. Number of iterations and CPU time for ADMM, AA, and aAA--FP methods. The notation $\dagger$ indicates that the solver did not achieve the desired stopping criterion within 1000 iterations.}\label{table:TV1}
 	\begin{scriptsize}
 		\begin{center}
 			{\begin{tabular}{|cc|cc|cc|cc|cc|}
 					\hline
 					\multicolumn{2}{|c|}{ADMM} & \multicolumn{2}{c|}{AA(1)} & \multicolumn{2}{c|}{AA(5)} & \multicolumn{2}{c|}{AA(10)} & \multicolumn{2}{c|}{AA($\infty$)} \\
 					\hline
 					$\texttt{it}$ & CPU & $\texttt{it}$ & CPU & $\texttt{it}$ & CPU & $\texttt{it}$ & CPU & $\texttt{it}$ & CPU \\
 					\cellcolor{gray}1000$\dagger$ & 332 & 582 & 215 & 218 & 79.7 & 150 & 54.9 & \cellcolor{pink} 91 & \cellcolor{green} 33.5 \\
 					\hline
 					\multicolumn{2}{|c|}{aAA(10)[1]--FP[10]} & \multicolumn{2}{c|}{aAA(1)[5]--FP[2]} & \multicolumn{2}{c|}{aAA(5)[1]--FP[5]} & \multicolumn{2}{c|}{aAA(10)[5]--FP[10]} & \multicolumn{2}{c|}{aAA($\infty$)[1]-FP[1]} \\
 					\hline
 					$\texttt{it}$ & CPU & $\texttt{it}$ & CPU & $\texttt{it}$ & CPU & $\texttt{it}$ & CPU & $\texttt{it}$ & CPU \\
 					122 & 46.1 & 144 & 54.6 & 133 & 49.1 & 106 & 38.7 & 92 & 37.4 \\
 					\hline
 					\multicolumn{2}{|c|}{aAA(3)[3]--FP[3]} & \multicolumn{2}{c|}{aAA(8)[3]--FP[7]} & \multicolumn{2}{c|}{aAA(5)[3]--FP[5]} & \multicolumn{2}{c|}{aAA(10)[10]--FP[10]} & \multicolumn{2}{c|}{aAA($\infty$)[2]--FP[1]} \\
 					\hline
 					$\texttt{it}$ & CPU & $\texttt{it}$ & CPU & $\texttt{it}$ & CPU & $\texttt{it}$ & CPU & $\texttt{it}$ & CPU \\
 					121 & 46.1 & 110 & 42.4 & 119 & 44.9 & \cellcolor{pink} 100 & \cellcolor{green} 36.5 &  94 & 34.1  \\
 					\hline
 					\multicolumn{2}{|c|}{aAA(5)[5]--FP[3]} & \multicolumn{2}{c|}{aAA(3)[10]--FP[5]} & \multicolumn{2}{c|}{aAA(5)[5]--FP[5]} & \multicolumn{2}{c|}{aAA(10)[15]--FP[10]} & \multicolumn{2}{c|}{aAA($\infty$)[5]--FP[2]}\\
 					\hline
 					$\texttt{it}$ & CPU & $\texttt{it}$ & CPU & $\texttt{it}$ & CPU & $\texttt{it}$ & CPU & $\texttt{it}$ & CPU \\
 					109 & 38.2 & 115 & 44.1 & 110 & 40.2 &  97 & 39.5 & 94 & 33.9 \\
 					\hline
 					\multicolumn{2}{|c|}{aAA(5)[7]--FP[5]} & \multicolumn{2}{c|}{aAA(8)[8]--FP[8]} & \multicolumn{2}{c|}{aAA(5)[8]--FP[5]} & \multicolumn{2}{c|}{aAA(10)[15]--FP[15]} & \multicolumn{2}{c|}{aAA($\infty$)[3]--FP[7]}\\
 					\hline
 					$\texttt{it}$ & CPU & $\texttt{it}$ & CPU & $\texttt{it}$ & CPU & $\texttt{it}$ & CPU & $\texttt{it}$ & CPU \\
 					103 & 44.6 & 106 & 37.1   & 101 & 36.5 & 109  & 39.7  & 109 & 41.4   \\
 					\hline
 					\multicolumn{2}{|c|}{aAA(5)[10]--FP[10]} & \multicolumn{2}{c|}{aAA(10)[10]--FP[5]} & \multicolumn{2}{c|}{aAA(5)[10]--FP[5]} & \multicolumn{2}{c|}{aAA(10)[15]--FP[20]} & \multicolumn{2}{c|}{aAA($\infty$)[10]--FP[10]}\\
 					\hline
 					$\texttt{it}$ & CPU & $\texttt{it}$ & CPU & $\texttt{it}$ & CPU & $\texttt{it}$ & CPU & $\texttt{it}$ & CPU \\
 					112 & 38.7 & 101 & 34.4 & 102 & 37.2 & 127  & 46.0  & 101 & 34.7   \\
 					\hline
 			\end{tabular}}
 		\end{center}
 	\end{scriptsize}
 \end{table}

 \begin{figure}[H]
 	\centering
 	\includegraphics[width=0.49\linewidth]{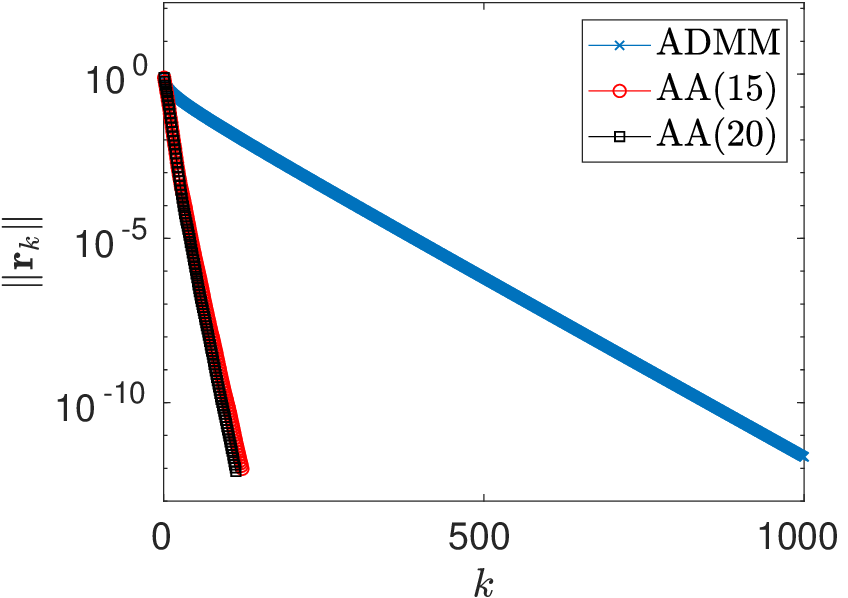}
 	\includegraphics[width=0.49\linewidth]{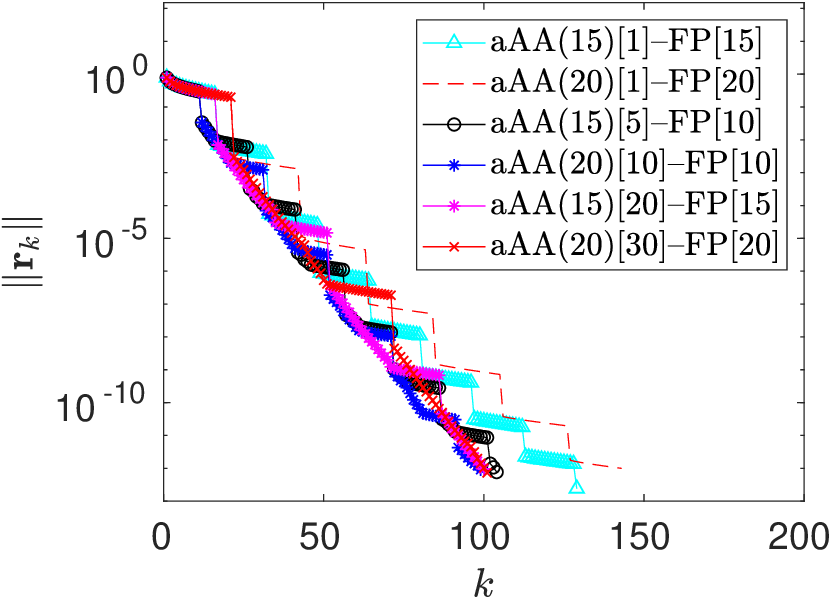}
 	\caption{Comparison of ADMM, AA($m$), and aAA($m$)[$s$]--FP[$t$], for different values of $m$, $s$, and $t$.}
 	\label{fig:TV}
 \end{figure}
 
 First, from Table \ref{table:TV1} we note that increasing $m$ results in a faster convergence of AA($m$), see first row.  Second, all aAA($m$)[$s$]--FP[$t$] considered here outperform AA($m$) in terms of both the number of iterations and CPU time; see the second, third and fourth columns. Especially, aAA(5)[5]--FP[5] converges in about half the number of iterations and CPU time compared to AA(5) for the problem considered here. Third, aAA($m$)[$s$]--FP[$t$] outperform aAA($m$)[1]--FP[$t$] in terms of both the number of iterations and CPU time, except for aAA($10$)[$15$]--FP[$10$] in terms of CPU time.   Fourth, our aAA($m$)[$s$]--FP[$t$] results in a faster convergence than ADMM, being more than 10 times faster in terms of both the number of iterations and CPU time. Finally, we note that for a careful choice of the parameter aAA($m$)[$s$]--FP[$t$] is comparable with AA($\infty$), see, for example aAA(10)[15]--FP[10]; we would like to mention that AA($\infty$) is only used as a reference solver and is not likely employed in practice for the prohibitive dimension of the least squares problem to be solved as the number of iterations grows. We also note that aAA($\infty$)[$s$]--FP[$t$] is as competitive as AA($\infty$), suggesting its potential usefulness in other applications.
 
 From Figure \ref{fig:TV} we observe the speed-up obtained by applying AA($m$) to accelerate ADMM; in fact, AA($m$) is 10 times faster than ADMM alone in terms of number of iterations. Furthermore, we observe that the number of iterations required by aAA($m$)[$s$]--FP[$t$] can be comparable to that required by AA($m$) to reach a prescribed tolerance, provided that the parameters $s$ and $t$ are appropriately chosen.

 \subsubsection{Lasso problem}
 We now consider the lasso problem, which is nonlinear and nonsmooth. Given the matrix $\mathbf{C}\in \mathbb{R}^{n_1 \times n_2}$ and the vector of data $\widehat{\mathbf{x}} \in \mathbb{R}^n_1$, the lasso problem is formulated as follows:
 \begin{displaymath}
 	\min_{\mathbf{x}} \dfrac{1}{2} \| \mathbf{C} \mathbf{x} - \widehat{\mathbf{x}} \|^2_2 + \beta \| \mathbf{x}\|_1,
 \end{displaymath}
 where $\beta>0$ is a regularization parameter. Again, we can reformulate this problem within the ADMM framework, by introducing the variable $\mathbf{y} = \mathbf{x}$ and considering
 \begin{displaymath}
 	\min_{\mathbf{x}, \mathbf{y}} \dfrac{1}{2} \| \mathbf{C} \mathbf{x}- \widehat{\mathbf{x}} \|^2_2 + \beta \| \mathbf{y}\|_1,
 \end{displaymath}
 such that
 \begin{displaymath}
 	\mathbf{x} - \mathbf{y} = \mathbf{0}.
 \end{displaymath}
 
 Following \cite{wang2021asymptotic}, we set $n_1=150$ and $n_2=300$, and take matrix $\mathbf{C}$ to be a randomly generated sparse matrix with density 0.01 and the vector $\widehat{\mathbf{x}}$ randomly generated, both sampled from the normal distribution. Further, we set the regularization parameter $\beta=1$ and the penalty parameter $\mu=10$. As above, we allow for 1000 iterations and set as stopping criterion a reduction of $10^{-12}$ in the residual.
 
 We first consider using the zero vector as the initial guess. We report the number of iterations and CPU time for ADMM, AA, and aAA($m$)[$s$]--FP[$t$] methods in Table \ref{table:lasso1-single}. 
 As we observed for the previous problem, alternating between AA and the fixed-point iteration can result in faster convergence of the method. In Table \ref{table:lasso1-single}, we observed that aAA($m$)[$s$]--FP[$t$] can be between 2 and 3 times faster than AA($m$), compare, for example, aAA(1)[3]--FP[3] and AA(1). More generally, by carefully selecting the parameters, one can obtain a solver that is comparable to AA($\infty$) in terms of both the number of iterations and CPU time; see, e.g., aAA(8)[10]--FP[3] in Table \ref{table:lasso1-single}.  Finally, aAA($m$)[$s$]--FP[$t$] can significantly speed up the performance of ADMM; for example, aAA($8$)[$10$]--FP[$3$] yields a solver that is over 4 times faster than ADMM in CPU time and requires only one-sixth the number of iterations.
 
Motivated by observations in the literature that initial guesses influence the performance of AA, we employ 50 different initial guesses generated by Matlab's rand function with seed 100,000 for the runs. Table \ref{table:lasso1-ave} reports the average number of iterations and CPU time. Although the average iteration count is higher than the single-run results of Table \ref{table:lasso1-single}, the average CPU time is lower than that obtained with a single initial guess. Overall, aAA--FP outperforms AA($m$) and ADMM. The results in Table \ref{table:lasso1-ave} also suggest that different initial guesses have a significant influence on the performance of ADMM, AA, and aAA--FP methods. We do not further investigate the influence of initial guesses on the performance of aAA--FP, which could be an interesting direction for future work.
 
From Tables \ref{table:lasso1-single} and \ref{table:lasso1-ave}, we notice that for fixed $m$ and $t$, increasing $s$ improves the performance, but once $s$ become sufficiently large, the gains diminish. In practice, it is advisable not to choose $s$ to be too large.
 
 \begin{table}[!ht]
 	\caption{Lasso problem. Number of iterations and CPU time for ADMM, AA, and aAA--FP methods using the zero vector as the initial guess.}\label{table:lasso1-single}
 	\begin{scriptsize}
 		\begin{center}
 			{\begin{tabular}{|cc|cc|cc|cc|cc|}
 					\hline
 					\multicolumn{2}{|c|}{ADMM} & \multicolumn{2}{c|}{AA(1)} & \multicolumn{2}{c|}{AA(3)} & \multicolumn{2}{c|}{AA(8)} & \multicolumn{2}{c|}{AA($\infty$)} \\
 					\hline
 					$\texttt{it}$ & CPU & $\texttt{it}$ & CPU & $\texttt{it}$ & CPU & $\texttt{it}$ & CPU & $\texttt{it}$ & CPU \\
 					\cellcolor{gray} 480 & 0.82 & 221 & 0.83 & 92 & 0.45 & 61 & 0.28 & \cellcolor{pink}47 & \cellcolor{green} 0.22 \\
 					\hline
 					\multicolumn{2}{|c|}{aAA(5)[1]--FP[5]} & \multicolumn{2}{c|}{aAA(1)[1]--FP[1]} & \multicolumn{2}{c|}{aAA(3)[1]--FP[3]} & \multicolumn{2}{c|}{aAA(8)[1]--FP[8]} & \multicolumn{2}{c|}{aAA($\infty$)[1]--FP[5]} \\
 					\cline{1-10}
 					$\texttt{it}$ & CPU & $\texttt{it}$ & CPU & $\texttt{it}$ & CPU & $\texttt{it}$ & CPU & $\texttt{it}$ & CPU \\
 					68& 0.25  & 85 & 0.43 & 85 & 0.35 & 72 & 0.36 & 143 & 0.66\\
 					\hline
 					\multicolumn{2}{|c|}{aAA(5)[3]--FP[5]} & \multicolumn{2}{c|}{aAA(1)[3]--FP[3]} & \multicolumn{2}{c|}{aAA(3)[2]--FP[3]} & \multicolumn{2}{c|}{aAA(8)[2]--FP[8]} & \multicolumn{2}{c|}{aAA($\infty$)[1]--FP[10]} \\
 					\cline{1-10}
 					$\texttt{it}$ & CPU & $\texttt{it}$ & CPU & $\texttt{it}$ & CPU & $\texttt{it}$ & CPU & $\texttt{it}$ & CPU \\
 					64 & 0.24 & 70 & 0.27 & 80 & 0.31 & 70 & 0.31 & 74 & 0.33 \\
 					\cline{1-10}
 					\multicolumn{2}{|c|}{aAA(10)[3]--FP[5]} & \multicolumn{2}{c|}{aAA(1)[5]--FP[2]} & \multicolumn{2}{c|}{aAA(3)[3]--FP[3]} & \multicolumn{2}{c|}{aAA(8)[3]--FP[7]} & \multicolumn{2}{c|}{aAA($\infty$)[2]--FP[7]} \\
 					\cline{1-10}
 					$\texttt{it}$ & CPU & $\texttt{it}$ & CPU & $\texttt{it}$ & CPU & $\texttt{it}$ & CPU & $\texttt{it}$ & CPU \\
 					63  & 0.25  & 83 & 0.33 & 71 & 0.27 & 69 & 0.29 & 104 & 0.37 \\
 					\cline{1-10}
 					\multicolumn{2}{|c|}{aAA(10)[5]--FP[5]} & \multicolumn{2}{c|}{aAA(1)[10]--FP[5]} & \multicolumn{2}{c|}{aAA(3)[5]--FP[3]} & \multicolumn{2}{c|}{aAA(8)[10]--FP[3]} & \multicolumn{2}{c|}{aAA($\infty$)[2]--FP[10]} \\
 					\cline{1-10}
 					$\texttt{it}$ & CPU & $\texttt{it}$ & CPU & $\texttt{it}$ & CPU & $\texttt{it}$ & CPU & $\texttt{it}$ & CPU \\
 					67 & 0.26  & 83 & 0.29  &69   &0.27   &  57 & 0.21 & 79  & 0.38  \\
 					\cline{1-10}
 					\multicolumn{2}{|c|}{aAA(10)[15]--FP[5]} & \multicolumn{2}{c|}{aAA(1)[15]--FP[5]} & \multicolumn{2}{c|}{aAA(3)[10]--FP[3]} & \multicolumn{2}{c|}{aAA(8)[15]--FP[3]} & \multicolumn{2}{c|}{aAA($\infty$)[10]--FP[10]} \\
 					\cline{1-10}
 					$\texttt{it}$ & CPU & $\texttt{it}$ & CPU & $\texttt{it}$ & CPU & $\texttt{it}$ & CPU & $\texttt{it}$ & CPU \\
 					\cellcolor{pink} 57 & \cellcolor{green} 0.14  & 93  &  0.12  & 62  &  0.08  & 59  & 0.08  & 104   & 0.21  \\
 					\cline{1-10}
 			\end{tabular}}
 		\end{center}
 	\end{scriptsize}
 \end{table}

 \begin{table}[!ht]
 	\caption{Lasso problem. The average of number of iterations and the average of CPU time for ADMM, AA, and aAA--FP methods using 50 random initial guesses  generated by Matlab's rand function with seed 100,000.}\label{table:lasso1-ave}
 	\begin{scriptsize}
 		\begin{center}
 			{\begin{tabular}{|cc|cc|cc|cc|cc|}
 					\hline
 					\multicolumn{2}{|c|}{ADMM} & \multicolumn{2}{c|}{AA(1)} & \multicolumn{2}{c|}{AA(3)} & \multicolumn{2}{c|}{AA(8)} & \multicolumn{2}{c|}{AA($\infty$)} \\
 					\hline
 					$\texttt{it}$ & CPU & $\texttt{it}$ & CPU & $\texttt{it}$ & CPU & $\texttt{it}$ & CPU & $\texttt{it}$ & CPU \\
 					\cellcolor{gray} 514 & 0.71 & 253 & 0.35 & 120 & 0.17 & 78 & 0.13 & 87 & 0.18 \\
 					\hline
 					\multicolumn{2}{|c|}{aAA(5)[1]--FP[5]} & \multicolumn{2}{c|}{aAA(1)[1]--FP[1]} & \multicolumn{2}{c|}{aAA(3)[1]--FP[3]} & \multicolumn{2}{c|}{aAA(8)[1]--FP[8]} & \multicolumn{2}{c|}{aAA($\infty$)[1]--FP[5]} \\
 					\cline{1-10}
 					$\texttt{it}$ & CPU & $\texttt{it}$ & CPU & $\texttt{it}$ & CPU & $\texttt{it}$ & CPU & $\texttt{it}$ & CPU \\
 					86 & 0.10 & 121 & 0.14 & 96 & 0.12 & 82 & 0.11 & 105 & 0.17 \\
 					\hline
 					\multicolumn{2}{|c|}{aAA(5)[3]--FP[5]} & \multicolumn{2}{c|}{aAA(1)[3]--FP[3]} & \multicolumn{2}{c|}{aAA(3)[2]--FP[3]} & \multicolumn{2}{c|}{aAA(8)[2]--FP[8]} & \multicolumn{2}{c|}{aAA($\infty$)[1]--FP[10]} \\
 					\cline{1-10}
 					$\texttt{it}$ & CPU & $\texttt{it}$ & CPU & $\texttt{it}$ & CPU & $\texttt{it}$ & CPU & $\texttt{it}$ & CPU \\
 					84 & 0.10 & 106 & 0.13 & 92 & 0.13 & 86 & 0.11 & 94 & 0.15 \\
 					\cline{1-10}
 					\multicolumn{2}{|c|}{aAA(10)[3]--FP[5]} & \multicolumn{2}{c|}{aAA(1)[5]--FP[2]} & \multicolumn{2}{c|}{aAA(3)[3]--FP[3]} & \multicolumn{2}{c|}{aAA(8)[3]--FP[7]} & \multicolumn{2}{c|}{aAA($\infty$)[2]--FP[7]} \\
 					\cline{1-10}
 					$\texttt{it}$ & CPU & $\texttt{it}$ & CPU & $\texttt{it}$ & CPU & $\texttt{it}$ & CPU & $\texttt{it}$ & CPU \\
 					77 & 0.10 & 100 & 0.12 & 87 & 0.10 & 83 & 0.12 & 143 & 0.56 \\
 					\cline{1-10}
 					\multicolumn{2}{|c|}{aAA(10)[5]--FP[5]} & \multicolumn{2}{c|}{aAA(1)[10]--FP[5]} & \multicolumn{2}{c|}{aAA(3)[5]--FP[3]} & \multicolumn{2}{c|}{aAA(8)[10]--FP[3]} & \multicolumn{2}{c|}{aAA($\infty$)[2]--FP[10]} \\
 					\cline{1-10}
 					$\texttt{it}$ & CPU & $\texttt{it}$ & CPU & $\texttt{it}$ & CPU & $\texttt{it}$ & CPU & $\texttt{it}$ & CPU \\
 					77 & 0.10 & 115 & 0.14 & 83 & 0.10 & \cellcolor{pink} 75 & \cellcolor{green} 0.09 & 125 & 0.33 \\
 					\cline{1-10}
 					\multicolumn{2}{|c|}{aAA(10)[15]--FP[5]} & \multicolumn{2}{c|}{aAA(1)[15]--FP[5]} & \multicolumn{2}{c|}{aAA(3)[10]--FP[3]} & \multicolumn{2}{c|}{aAA(8)[15]--FP[3]} & \multicolumn{2}{c|}{aAA($\infty$)[10]--FP[10]} \\
 					\cline{1-10}
 					$\texttt{it}$ & CPU & $\texttt{it}$ & CPU & $\texttt{it}$ & CPU & $\texttt{it}$ & CPU & $\texttt{it}$ & CPU \\
 					78 &  0.12  & 116  & 0.17   & 87  &  0.13  &  77 &  0.12 & 385   &  4.2 \\
 					\cline{1-10}
 			\end{tabular}}
 		\end{center}
 	\end{scriptsize}
 \end{table}

 \subsubsection{Nonnegative least squares problem}
 As a final example of accelerating ADMM, we consider the following nonnegative least squares problem:
 \begin{displaymath}
 	\min_{\mathbf{x}} \| \mathbf{C} \mathbf{x} - \widehat{\mathbf{x}} \|_2^2,
 \end{displaymath}
 constrained with $\mathbf{x} \geq \mathbf{0}$. Here, the matrix $\mathbf{C}\in \mathbb{R}^{n_1 \times n_2}$ and the vector $\widehat{\mathbf{x}}\in \mathbb{R}^{n_1}$ are given. The nonnegative least squares problem is an example of nonlinear problem with inequality constraints; see, e.g., \cite{Fu_Zhang_Boyd}.
 
 In order to employ ADMM, we introduce the indicator function $\mathcal{I}_{\mathbb{R}^{n_2}_+}(\cdot)$ given by
 \begin{displaymath}
 	\mathcal{I}_{\mathbb{R}^{n_2}_+}(\mathbf{y}) :=
 	\left\{
 	\begin{array}{ll}
 		0 & \mathrm{if} \; \mathbf{y} \geq 0, \\
 		+\infty & \mathrm{otherwise}.
 	\end{array}
 	\right.
 \end{displaymath}
 Then, after introducing the variable $\mathbf{y}= \mathbf{x}$, we can reformulate the nonlinear least squares problems as follows:
 \begin{displaymath}
 	\min_{\mathbf{x}, \mathbf{y}} \|\mathbf{C} \mathbf{x} - \widehat{\mathbf{x}} \|_2^2 + \mathcal{I}_{\mathbb{R}^{n_2}_+}(\mathbf{y}),
 \end{displaymath}
 constrained with
 \begin{displaymath}
 	\mathbf{x} - \mathbf{y}=\mathbf{0}.
 \end{displaymath}
 
 For this example, we set $n_1=150$, $n_2=300$, and take $\mathbf{C}$ as a randomly generated sparse matrix with density 0.01. Further, the vector $\widehat{\mathbf{x}}$ is randomly generated and sampled from the normal distribution. Finally, we set the penalty parameter $\mu=2$. For this test, we look for a reduction of $10^{-12}$ in the residual, allowing for a maximum of 2000 iterations. 
 
In Table \ref{table:nnls1-single}, we report number of iterations and CPU time for ADMM, AA, and aAA--FP methods using the zero vector as the initial guess.  Then, we run the solvers for 50 randomly generated initial guesses that are generated by Matlab's rand function with seed 100,000, and report the average number of iterations and the average CPU time in Table \ref{table:nnls1-ave}. Overall, the average CPU time is lower than that obtained with a single initial guess.

 \begin{table}[!ht]
 	\caption{Nonnegative least squares problem. Number of iterations and CPU time for ADMM, AA, and aAA--FP methods using the zero vector as the initial guess. The notation $\dagger$ indicates that the solver did not achieve the desired stopping criterion within 2000 iterations.}\label{table:nnls1-single}
 	\begin{scriptsize}
 		\begin{center}
 			{\begin{tabular}{|cc|cc|cc|cc|cc|}
 					\hline
 					\multicolumn{2}{|c|}{ADMM} & \multicolumn{2}{c|}{AA(1)} & \multicolumn{2}{c|}{AA(3)} & \multicolumn{2}{c|}{AA(10)} & \multicolumn{2}{c|}{AA($\infty$)} \\
 					\hline
 					$\texttt{it}$ & CPU & $\texttt{it}$ & CPU & $\texttt{it}$ & CPU & $\texttt{it}$ & CPU & $\texttt{it}$ & CPU \\
 					\cellcolor{gray} 2000$\dagger$ & 2.43 & 77 & 0.27 & 52 & 0.19 & 38 &  0.13 & \cellcolor{pink} 30 & \cellcolor{green} 0.13 \\
 					\hline
 					\multicolumn{2}{|c|}{aAA(5)[1]--FP[5]} & \multicolumn{2}{c|}{aAA(1)[1]--FP[1]} & \multicolumn{2}{c|}{aAA(3)[1]--FP[3]} & \multicolumn{2}{c|}{aAA(10)[1]--FP[10]} & \multicolumn{2}{c|}{aAA($\infty$)[1]--FP[5]} \\
 					\hline
 					$\texttt{it}$ & CPU & $\texttt{it}$ & CPU & $\texttt{it}$ & CPU & $\texttt{it}$ & CPU & $\texttt{it}$ & CPU \\
 					43&  0.18 &  85 & 0.35  & 85  & 0.31 &45   &0.19   &37   &0.15   \\
 					\hline
 					\multicolumn{2}{|c|}{aAA(5)[8]--FP[5]} & \multicolumn{2}{c|}{aAA(1)[3]--FP[1]} & \multicolumn{2}{c|}{aAA(3)[3]--FP[3]} & \multicolumn{2}{c|}{aAA(10)[10]--FP[10]} & \multicolumn{2}{c|}{aAA($\infty$)[1]--FP[10]}  \\
 					\hline
 					$\texttt{it}$ & CPU & $\texttt{it}$ & CPU & $\texttt{it}$ & CPU & $\texttt{it}$ & CPU & $\texttt{it}$ & CPU \\
 					37 & 0.14   & 85  & 0.34  & 47 & 0.18 & \cellcolor{pink} 34 & \cellcolor{green} 0.15 & 34 & 0.14 \\
 					\hline
 					\multicolumn{2}{|c|}{aAA(8)[10]--FP[3]} & \multicolumn{2}{c|}{aAA(1)[8]--FP[1]} & \multicolumn{2}{c|}{aAA(3)[8]--FP[8]} & \multicolumn{2}{c|}{aAA(10)[10]--FP[5]} & \multicolumn{2}{c|}{aAA($\infty$)[1]--FP[15]} \\
 					\hline
 					$\texttt{it}$ & CPU & $\texttt{it}$ & CPU & $\texttt{it}$ & CPU & $\texttt{it}$ & CPU & $\texttt{it}$ & CPU \\
 					36 & 0.14 & 84  & 0.32 & 46  & 0.17 & 60 & 0.21  & 84 & 0.41  \\
 					\hline
 					\multicolumn{2}{|c|}{aAA(15)[10]--FP[5]} & \multicolumn{2}{c|}{aAA(1)[10]--FP[1]} & \multicolumn{2}{c|}{aAA(3)[10]--FP[5]} & \multicolumn{2}{c|}{aAA(10)[15]--FP[10]} & \multicolumn{2}{c|}{aAA($\infty$)[2]--FP[10]} \\
 					\hline
 					$\texttt{it}$ & CPU & $\texttt{it}$ & CPU & $\texttt{it}$ & CPU & $\texttt{it}$ & CPU & $\texttt{it}$ & CPU \\
 					57  & 0.21  & 94  & 0.33  & 73  & 0.29  & 63  & 0.23  & 79  & 0.34 \\
 					\hline
 			\end{tabular}}
 		\end{center}
 	\end{scriptsize}
 \end{table}

 \begin{table}[!ht]
 	\caption{Nonnegative least squares problem. The average of number of iterations and the average CPU time for ADMM, AA, and aAA--FP methods using 50 random initial guesses generated by Matlab's rand function with seed 100,000. The notation $\dagger$ indicates that the solver did not achieve the desired stopping criterion within 2000 iterations.}\label{table:nnls1-ave}
 	\begin{scriptsize}
 		\begin{center}
 			{\begin{tabular}{|cc|cc|cc|cc|cc|}
 					\hline
 					\multicolumn{2}{|c|}{ADMM} & \multicolumn{2}{c|}{AA(1)} & \multicolumn{2}{c|}{AA(3)} & \multicolumn{2}{c|}{AA(10)} & \multicolumn{2}{c|}{AA($\infty$)} \\
 					\hline
 					$\texttt{it}$ & CPU & $\texttt{it}$ & CPU & $\texttt{it}$ & CPU & $\texttt{it}$ & CPU & $\texttt{it}$ & CPU \\
 					\cellcolor{gray} 2000$\dagger$ & 1.73 & 90 & 0.10 & 61 & 0.07 & \cellcolor{pink} 43 & \cellcolor{green} 0.06 & 60 & 1.52 \\
 					\hline
 					\multicolumn{2}{|c|}{aAA(5)[1]--FP[5]} & \multicolumn{2}{c|}{aAA(1)[1]--FP[1]} & \multicolumn{2}{c|}{aAA(3)[1]--FP[3]} & \multicolumn{2}{c|}{aAA(10)[1]--FP[10]} & \multicolumn{2}{c|}{aAA($\infty$)[1]--FP[5]} \\
 					\hline
 					$\texttt{it}$ & CPU & $\texttt{it}$ & CPU & $\texttt{it}$ & CPU & $\texttt{it}$ & CPU & $\texttt{it}$ & CPU \\
 					51 & 0.05 & 92 & 0.10 & 53 & 0.06 & 45 & 0.05 & 68 & 0.20 \\
 					\hline
 					\multicolumn{2}{|c|}{aAA(5)[8]--FP[5]} & \multicolumn{2}{c|}{aAA(1)[3]--FP[1]} & \multicolumn{2}{c|}{aAA(3)[3]--FP[3]} & \multicolumn{2}{c|}{aAA(10)[10]--FP[10]} & \multicolumn{2}{c|}{aAA($\infty$)[1]--FP[10]}  \\
 					\hline
 					$\texttt{it}$ & CPU & $\texttt{it}$ & CPU & $\texttt{it}$ & CPU & $\texttt{it}$ & CPU & $\texttt{it}$ & CPU \\
 					46 & 0.05 & 61 & 0.07 & 55 & 0.06 & \cellcolor{pink} 43 & \cellcolor{green} 0.06 & 46 & 0.06 \\
 					\hline
 					\multicolumn{2}{|c|}{aAA(8)[10]--FP[3]} & \multicolumn{2}{c|}{aAA(1)[8]--FP[1]} & \multicolumn{2}{c|}{aAA(3)[8]--FP[8]} & \multicolumn{2}{c|}{aAA(10)[10]--FP[5]} & \multicolumn{2}{c|}{aAA($\infty$)[1]--FP[15]} \\
 					\hline
 					$\texttt{it}$ & CPU & $\texttt{it}$ & CPU & $\texttt{it}$ & CPU & $\texttt{it}$ & CPU & $\texttt{it}$ & CPU \\
 					44 & 0.05 & 90 & 0.09 & 51 & 0.06 & 46 & 0.06 & 173 & 0.69 \\
 					\hline
 					\multicolumn{2}{|c|}{aAA(15)[10]--FP[5]} & \multicolumn{2}{c|}{aAA(1)[10]--FP[1]} & \multicolumn{2}{c|}{aAA(3)[10]--FP[5]} & \multicolumn{2}{c|}{aAA(10)[15]--FP[10]} & \multicolumn{2}{c|}{aAA($\infty$)[2]--FP[10]} \\
 					\hline
 					$\texttt{it}$ & CPU & $\texttt{it}$ & CPU & $\texttt{it}$ & CPU & $\texttt{it}$ & CPU & $\texttt{it}$ & CPU \\
 					43 & 0.06 & 93 & 0.10 & 52 & 0.06 & 43 & 0.06 & 47 & 0.06 \\
 					\hline
 			\end{tabular}}
 		\end{center}
 	\end{scriptsize}
 \end{table}

 In Table \ref{table:nnls1-single}, we note the beneficial effect of aAA($m$)[$s$]--FP[$t$] over the simple ADMM. First, we note that AA($m$) is more than 10 times faster than ADMM, in terms of both number of iterations and CPU time, see, for example, AA(10). Furthermore, we observe that as $m$ increases, the number of iterations of both AA and aAA--FP decreases. However, no substantial gain is obtained for very large values of $m$; compare, for instance, AA(10) with AA($\infty$). We observe that, by carefully choosing the parameters $m, s, t$, one can construct an aAA($m$)[$s$]--FP[$t$] solver that outperforms AA in terms of both number of iterations and CPU times. For example, in Table \ref{table:nnls1-ave}, aAA(3)[8]--FP[8] achieves the prescribed residual reduction in fewer iterations than AA(3), while requiring a CPU time comparable to that of AA(10). Finally, as with the previous problems, we observe that aAA($m$)[$s$]--FP[$t$] can significantly accelerate the performance of ADMM, resulting in a solver that is over 50 times faster than ADMM in iteration count and 34 times faster in CPU time; see, for example, aAA(5)[8]--FP[5] or aAA(8)[10]--FP[3] in Table \ref{table:nnls1-ave}.
 
 In Figure \ref{fig:NNLS-res-story} we plot the history of the residuals for aAA(10)[10]--FP(10), for the 50 different initial guesses. As observed, for some initial guesses the solver requires more than 50 iterations in order to reach convergence. This is not surprising, as the behavior of the classical Anderson acceleration strongly depends on the initial guess; see \cite{de2024anderson}. Nonetheless, the solver requires in average 43 iterations for reaching the prescribed reduction in the residual; see Table \ref{table:nnls1-ave}.
 
 \begin{figure}[H]
 	\centering
 	\includegraphics[width=0.6\linewidth]{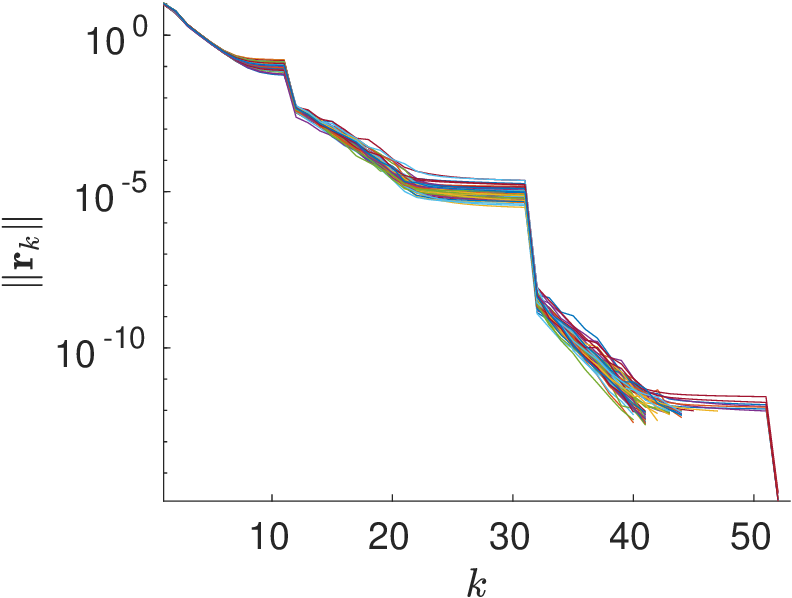}
 	\caption{Nonnegative least squares problem. History of residuals for aAA(10)[10]--FP(10) for 50 random initial guesses generated by Matlab's rand function with seed 100,000.}
 	\label{fig:NNLS-res-story}
 \end{figure}

 We would like to mention that for AA($\infty$) Matlab failed to solve the least-squares problem \eqref{eq:min-AA} for some of the instances because the matrix of the least-squares problem was found to be rank-deficient. In some cases, the first 15 columns (each column has the form $r(\mathbf{x}_k)-r(\mathbf{x}_{k-i}$)) have very small norms, on the order of $10^{-9}$ or $10^{-10}$, which results in an effective loss of rank. In the literature, many efforts have been made to address how to efficiently and robustly solve the least-squares problems involved in AA, for example \cite{lockhart2022performance,loffeld2016considerations}. We leave such implementation details for future work. In practice, we do not recommend using AA($\infty$). For aAA($m$)[$s$]--FP[$t$], although the choice of $m$ is often problem-dependent, it should remain modest in size--a few tens--to ensure efficiency and stability.

 \subsection{Regularized logistic regression}
 In the final numerical test, we examine the regularized logistic regression problem. This is defined as the minimization of the following objective function:
 \begin{displaymath}
 	\min_{\mathbf{x}} ~ h(\mathbf{x}) := \dfrac{1}{n_1} \sum_{i=1}^{n_1} \log(1 + \exp(-y_i \mathbf{x}^\top \mathbf{c}_i)) + \dfrac{\beta}{2} \| \mathbf{x} \|^2_2,
 \end{displaymath}
 and 
 \begin{displaymath}
 	\mathbf{C} = \left[
 	\begin{array}{c}
 		\mathbf{c}_1^T \\
 		\vdots \\
 		\mathbf{c}_{n_1}^T
 	\end{array}    
 	\right] \in \mathbb{R}^{n_1 \times n_2}
 \end{displaymath}
 is the matrix containing the data samples, with $y_i$ the corresponding label for $i=1,\ldots, n_1$, and $\beta>0$ is a regularization parameter. For this example, we follow the work in \cite{feng2024convergence} and apply gradient descent (GD) to minimize the objective function. Within this framework, starting from an approximation $\mathbf{x}_k$, we take a step in the direction of the negative gradient of $h$, scaled by a constant. Specifically, we consider the following recurrence:
 \begin{displaymath}
 	\mathbf{x}_{k+1} = \mathbf{x}_k - \eta \nabla h(\mathbf{x}_k),
 \end{displaymath}
 where $\eta$ is the step length taken. We test how aAA($m$)[$s$]--FP[$t$] can accelerate the convergence of GD for the (scaled) dataset \texttt{covtype} of the LIBSVM library \cite{LIBSVM} available at \url{http://www.csie.ntu.edu.tw/~cjlin/libsvm}; for this test set, we have $n_1=581,012$ and $n_2= 54$. We would like to mention that a study of how AA($m$) and aAA($m$)[1]--FP[$m$] (referred to as AAP($m$) in \cite{feng2024convergence}) accelerate GD for this problem was performed in \cite{feng2024convergence}. The authors in \cite{feng2024convergence} observed that the alternating Anderson--Picard presents a staircase pattern, being able to significantly reduce the error at each periodic AA; they observed that AAP(5) and AAP(7) converge to a tolerance of $10^{-10}$ in about 80 Picard iterations for the problem considered here, where they did not count the AA step in each period. Thus, AAP(5) and AAP(7) take around 96 and 91 iterations, respectively. We run GD, AA, and aAA--FP for 1000 iterations, seeking a reduction of $10^{-12}$ in the residual. As in \cite{feng2024convergence}, we set $\beta=10^{-2}$ and $\eta=1$. In Table \ref{table:logistic}, we report the number of iterations and the CPU time required to reach the prescribed tolerance. For this test, we employ the zero vector as the initial guess.

 \begin{table}[!ht]
 	\caption{Regularized logistic regression. Number of iterations and CPU time for GD, AA, and aAA--FP methods, when employing the zero vector as the initial guess. The notation $\dagger$ indicates that the solver did not achieve the desired stopping criterion within 1000 iterations.}\label{table:logistic}
 	\begin{scriptsize}
 		\begin{center}
 			{\begin{tabular}{|cc|cc|cc|cc|cc|}
 					\hline
 					\multicolumn{2}{|c|}{GD} & \multicolumn{2}{c|}{AA(1)} & \multicolumn{2}{c|}{AA(3)} & \multicolumn{2}{c|}{AA(10)} & \multicolumn{2}{c|}{AA($\infty$)} \\
 					\hline
 					$\texttt{it}$ & CPU & $\texttt{it}$ & CPU & $\texttt{it}$ & CPU & $\texttt{it}$ & CPU & $\texttt{it}$ & CPU \\
 					\cellcolor{gray} 1000$\dagger$ & 1122 & \cellcolor{gray} 1000$\dagger$ & 1129 & 110 & 124 & 75 & 84.9 & \cellcolor{gray} 1000$\dagger$ & 1138 \\
 					\hline
 					\multicolumn{2}{|c|}{AA(7)} & \multicolumn{2}{c|}{AA(5)} & \multicolumn{2}{c|}{aAA(3)[1]--FP[3]} & \multicolumn{2}{c|}{aAA(10)[1]--FP[5]} & \multicolumn{2}{c|}{aAA(7)[3]--FP[3]} \\
 					\hline
 					$\texttt{it}$ & CPU & $\texttt{it}$ & CPU & $\texttt{it}$ & CPU & $\texttt{it}$ & CPU & $\texttt{it}$ & CPU \\
 					64 & 72.3 & 85&  93.4 & 41 & 45.7 & 43 & 46.9 & 41 & 46.7 \\
 					\hline
 					\multicolumn{2}{|c|}{aAA(7)[5]--FP[3]} & \multicolumn{2}{c|}{aAA(6)[3]--FP[3]} & \multicolumn{2}{c|}{aAA(3)[1]--FP[5]} & \multicolumn{2}{c|}{aAA(10)[3]--FP[5]} & \multicolumn{2}{c|}{aAA(5)[3]--FP[3]}  \\
 					\hline
 					$\texttt{it}$ & CPU & $\texttt{it}$ & CPU & $\texttt{it}$ & CPU & $\texttt{it}$ & CPU & $\texttt{it}$ & CPU \\
 					53 & 60.6 & 47 & 54.3 & 55 & 62.9 & 47 & 53.3 & 54 & 60.6 \\
 					\hline
 					\multicolumn{2}{|c|}{aAA(7)[10]--FP[3]} & \multicolumn{2}{c|}{aAA(6)[10]--FP[3]} & \multicolumn{2}{c|}{aAA(3)[3]--FP[3]} & \multicolumn{2}{c|}{aAA(10)[5]--FP[5]} & \multicolumn{2}{c|}{aAA(5)[5]--FP[3]} \\
 					\hline
 					$\texttt{it}$ & CPU & $\texttt{it}$ & CPU & $\texttt{it}$ & CPU & $\texttt{it}$ & CPU & $\texttt{it}$ & CPU \\
 					57 & 62.5 & 57 & 64.2 & 59 & 64.5 & 57 & 65.1 & 53 & 58.7 \\
 					\hline
 					\multicolumn{2}{|c|}{aAA(5)[1]--FP[2]} & \multicolumn{2}{c|}{aAA(5)[5]--FP[5]} & \multicolumn{2}{c|}{aAA(3)[6]--FP[3]} & \multicolumn{2}{c|}{aAA(5)[5]--FP[10]} & \multicolumn{2}{c|}{aAA(5)[10]--FP[3]} \\
 					\hline
 					$\texttt{it}$ & CPU & $\texttt{it}$ & CPU & $\texttt{it}$ & CPU & $\texttt{it}$ & CPU & $\texttt{it}$ & CPU \\
 					\cellcolor{pink} 28 & \cellcolor{green} 30.6 & 37 & 40.6 & 45 & 51.9 & 57 & 62.1 & 57 & 62.6 \\
 					\hline
 			\end{tabular}}
 		\end{center}
 	\end{scriptsize}
 \end{table}

 First, from Table \ref{table:logistic} we observe that as $m$ increases, the number of iterations of AA($m$) decreases, see the first row. We note, however, that in this case, finite windowed Anderson outperforms AA($\infty$), with the latter failing to converge within 1000 iterations. Further, we observe a clear benefit in applying AA to GD; in fact, the solver is over ten times faster in terms of both number of iterations and CPU time. Second, we note that aAA--FP results in a solver that is able to outperform AA($m$). For example, aAA(3)[6]--FP[3] is about two times faster than AA(3) in terms of number of iterations and CPU time. Finally, we would like to mention that by carefully choosing the parameters, $m, s, t$, one is able to obtain a solver that is nearly 50 times faster than GD alone and about 3 times faster than AA($m$), compare, for example, aAA(5)[1]--FP[2] with GD and AA(5). We conclude that the aAA--FP method is preferable.
 
 Regarding the choice of parameters, from Table \ref{table:logistic} we observe that $m \geq 2t$ usually results in a more efficient solver than AA alone for the problem considered here, a feature already observed in \cite{banerjee2016periodic} for aAA($m$)[$1$]--FP[$t$]. For example, compare aAA(7)[5]--FP[3] with AA(7).
 
 \subsection{Study the parameters choice}
 We study now the choice of the parameters, $m,s$ and $t$ in aAA($m$)[$s$]--FP[$t$]. We consider the total variation problem described in Section \ref{sec:TV}, and run aAA($m$)[$s$]--FP[$t$], employing zero as the initial guess. We report the number of iterations and the CPU times in second for three different cases: 
 	\begin{enumerate}
 		\item[(1.)] set $m=s$ and let $m, t \in \{1, \ldots, 20 \}$; 
 		\item[(2.)] set $m=t$ and let $m, s \in \{1, \ldots, 20 \}$;
 		\item[(3.)] set $s=t$ and let $m, t \in \{1, \ldots, 20 \}$.
 	\end{enumerate}
 	
 	In Figures \ref{fig:TV-contour1}, \ref{fig:TV-contour2}, and \ref{fig:TV-contour3} we report the contour plots of the number of iterations and the CPU times. In \cite{banerjee2016periodic}, it was found that taking $m\geq 2t$ generally improves performance of aAA($m$)[1]--FP[$t$] for the problems considered. Accordingly, we have also included the line $m=2t$ in the contour plots.  First, from Figure \ref{fig:TV-contour1} we observed that in both terms of total number of aAA($m$)[$m$]-FP[$t$] iterations and CPU time it is more beneficial increasing the number of Anderson iterations $s$ ($m=s$) than increasing the number of fixed point iterations $t$. Second, Figure \ref{fig:TV-contour2} shows that increasing the number of Anderson iterations ($s$) results in a lower number of aAA($m$)[$s$]-FP[$m$] iterations and computational time than increasing the number of fixed point iterations ($t=m$). Third, as observed in Figure \ref{fig:TV-contour3} one can obtain a reduction in the number of aAA($m$)[$t$]-FP[$t$] iterations and in the computational time by letting the window size $m$ increase.  We remark the best choice of the parameters are inside the dark blue area. Figures \ref{fig:TV-contour1} and \ref{fig:TV-contour3} seem suggest that using $m\geq 2t$ gives a better performance. 
 	
 	The idea behind aAA($m$)[$s$]--FP[$t$] is to apply 
 	$t$ fixed‑point iterations between every $s$ AA($m$) updates in order to reduce the computational cost of the least‑squares problems. However, if  $t$ is too large, the method behaves more like a standard fixed‑point iteration and the convergence may slow down. Thus, there is a trade‑off between iteration count and CPU time. In general,  we recommend that $t$ not be chosen significantly larger than the parameters $m$ and $s$, since these have a more substantial impact on performance.
 
 \begin{figure}[H]
 	\centering
 	\includegraphics[width=0.46\linewidth]{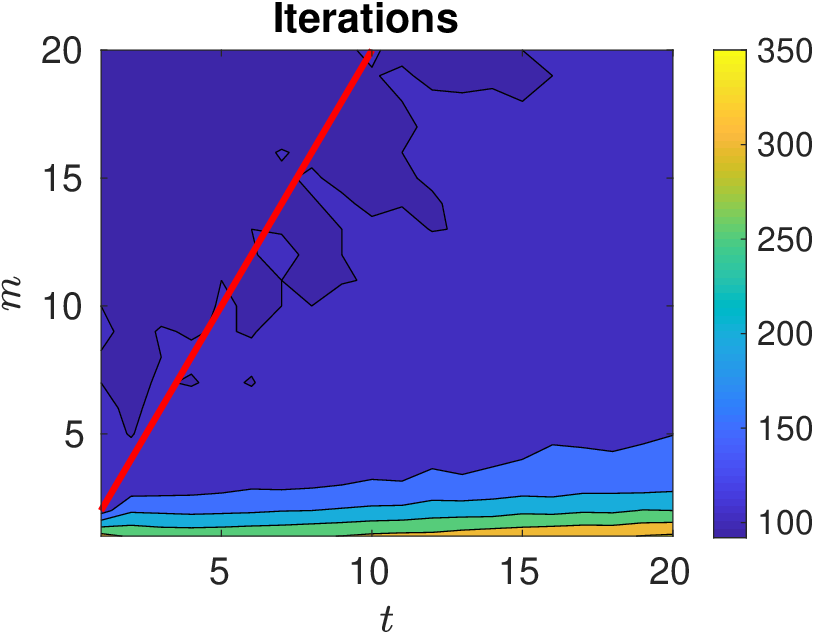}
 	\includegraphics[width=0.46\linewidth]{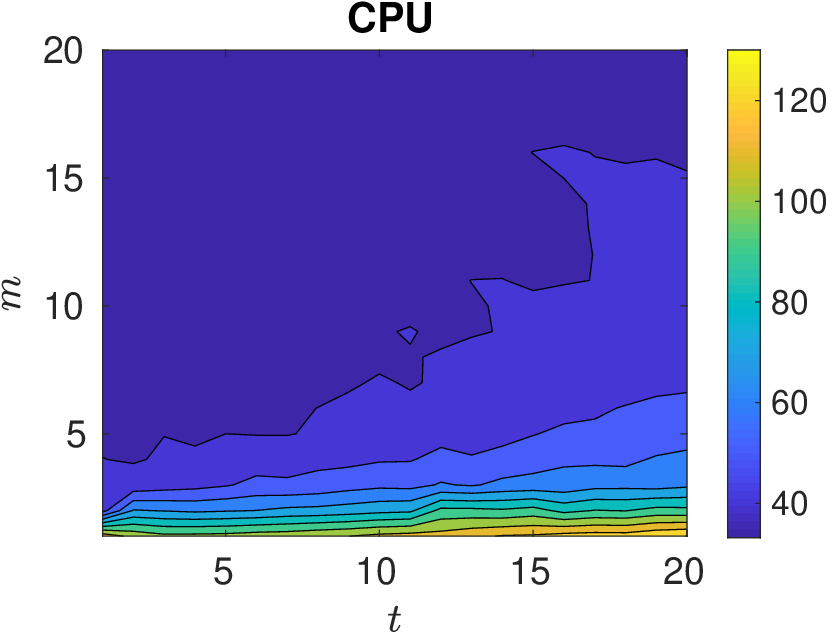}
 	\caption{Total variation. Contour plot of the number of iterations (left) and of the CPU times (right), when applying aAA($m$)[$m$]-FP[$t$], for $m,t \in \{1, \ldots, 20 \}$. The red line corresponds to $m=2t$.}
 	\label{fig:TV-contour1}
 \end{figure}
 
 \begin{figure}[H]
 	\centering
 	\includegraphics[width=0.46\linewidth]{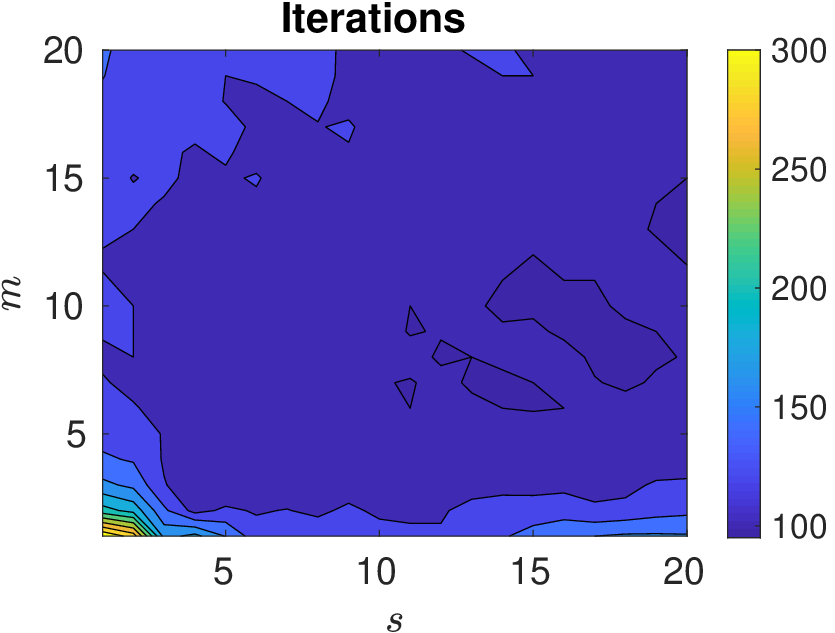}
 	\includegraphics[width=0.46\linewidth]{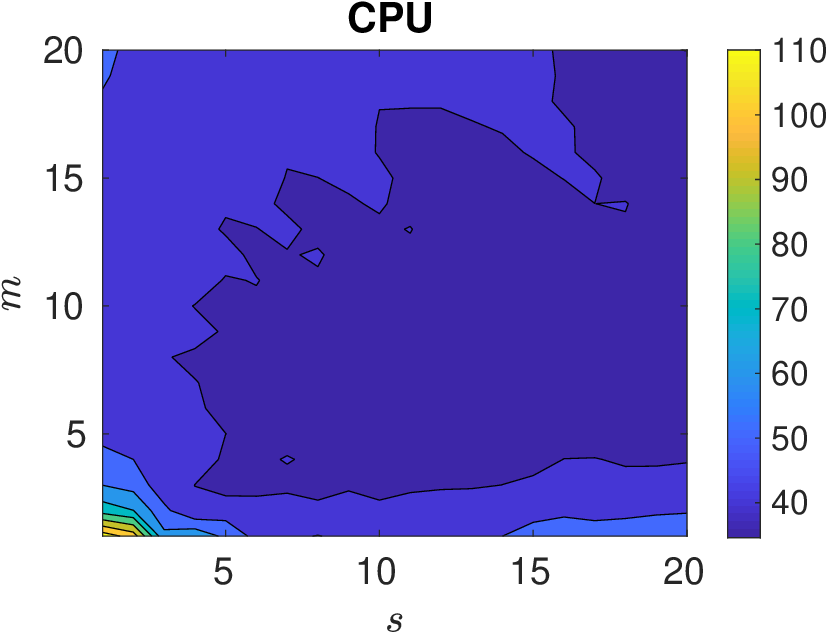}
 	\caption{Total variation. Contour plot of the number of iterations (left) and of the CPU times (right), when applying aAA($m$)[$s$]-FP[$m$], for $m,s \in \{1, \ldots, 20 \}$. }
 	\label{fig:TV-contour2}
 \end{figure}
 
 \begin{figure}[H]
 	\centering
 	\includegraphics[width=0.46\linewidth]{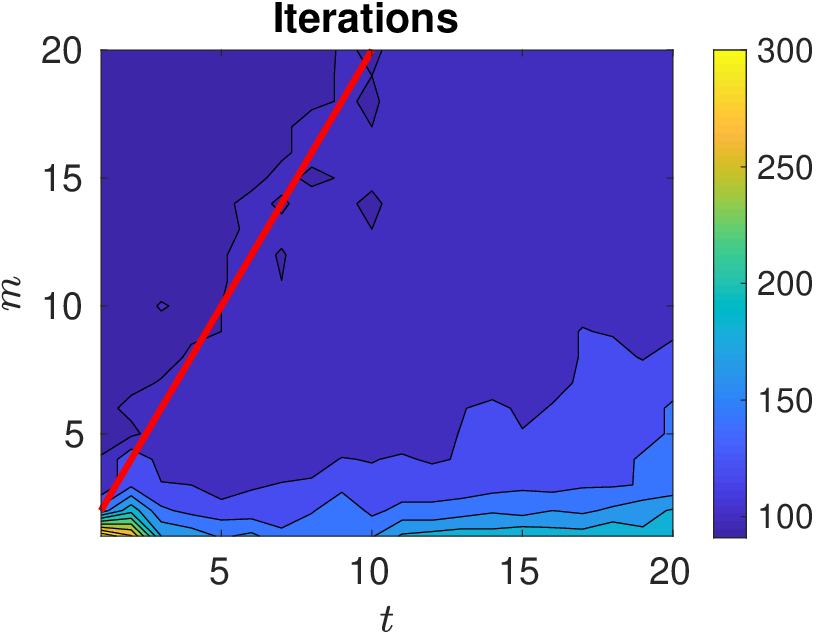}
 	\includegraphics[width=0.46\linewidth]{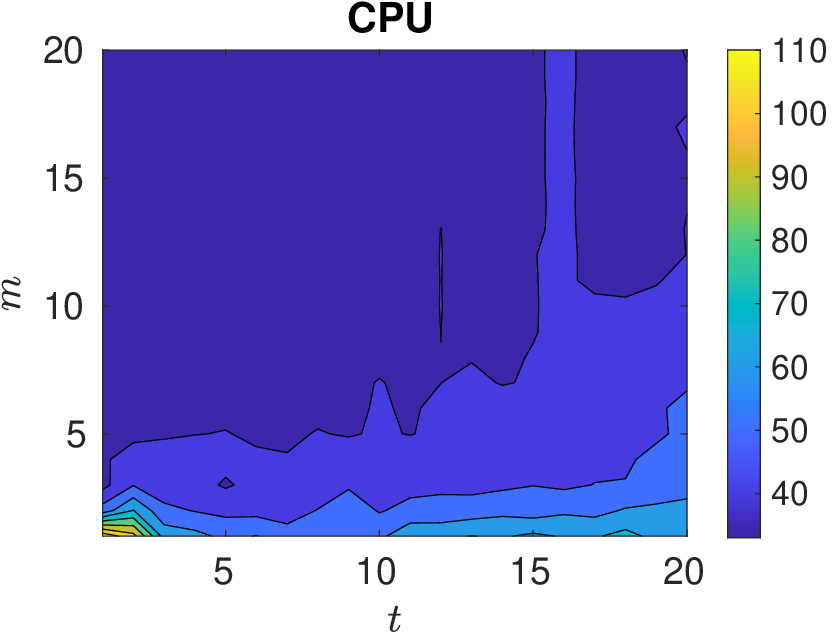}
 	\caption{Total variation. Contour plot of the number of iterations (left) and of the CPU times (right), when applying aAA($m$)[$t$]-FP[$t$], for $m,t \in \{1, \ldots, 20 \}$.  The red line corresponds to $m=2t$.}
 	\label{fig:TV-contour3}
 \end{figure}

 \section{Conclusion}\label{sec:con}
 
 In this work, we proposed a generalization of the alternating Anderson acceleration iteration, aAA($m$)[$s$]--FP[$t$]. The proposed method consists of a periodic pattern of $t$ fixed point iterations followed by $s$ iterations of AA($m$).   We provided a theoretical analysis of the convergence of the method, when applied for accelerating a linear fixed-point iteration. We established a connection between aAA($\infty$)[1]--FP[$t$] and GMRES in the linear case. Further, we gave a sufficient condition for the convergence of aAA--FP, when the fixed point iteration matrix is diagonalizable and noncontractive. Numerical results validated our theoretical findings, showing a periodic equivalence between GMRES and aAA($\infty$)[1]--FP[$t$]. In addition, we applied the proposed aAA--FP method to accelerate Jacobi iteration, Gauss--Seidel iteration, Picard iteration, gradient descent, and the alternating direction method of multipliers for the solutions of problems arising in fluid dynamics and nonlinear optimization. The results demonstrated that aAA--FP can drastically accelerate the fixed-point iterations and achieve competitive performance with AA($\infty$), in terms of both number of iterations and CPU times, provided that $m$, $s$, and $t$ are appropriately selected.
 
 From the numerical tests, we observed that the least squares problem to be solved within AA may be rank-deficient for large window size $m$.  Future work will focus on the development of efficient solvers for the least-squares problem that arises at each Anderson iteration. In addition, we mention that the code we employ in this work is not optimized. For this reason, future work will be devoted to devising Python and C++ implementations of the proposed strategy. Finally, we plan to explore the application of aAA--FP to more complex, real-world problems.

 \section*{Acknowledgements}
 S.L. is a member of Gruppo Nazionale di Calcolo Scientifico (GNCS) of the Istituto Nazionale di Alta Matematica (INdAM).

\bibliographystyle{plain}
\bibliography{aAAFPbib}

\end{document}